\documentclass[reqno]{amsart}

\usepackage{amsfonts}
\usepackage{mathrsfs}
\usepackage{amsmath}
\usepackage{amsmath,graphics}
\usepackage{amssymb}
\usepackage{indentfirst,latexsym,bm,amsmath,pstricks,amssymb,amsthm,graphicx,fancyhdr,float,color}
\usepackage[all]{xy}
\usepackage{amsthm}
\usepackage{amscd}
\usepackage{hyperref}
\usepackage[all]{hypcap}

\newtheorem{theorem}{Theorem}[section]
\newtheorem{lemma}[theorem]{Lemma}
\newtheorem{conjecture}[theorem]{Conjecture}
\newtheorem{problem}[theorem]{Problem}
\newtheorem{corollary}[theorem]{Corollary}
\newtheorem{proposition}[theorem]{Proposition}

\theoremstyle{definition}
\newtheorem{definition}[theorem]{Definition}
\newtheorem{example}[theorem]{Example}

\theoremstyle{remark}
\newtheorem{remark}[theorem]{Remark}
\newtheorem{remarks}[theorem]{Remarks}

\numberwithin{equation}{section}
\renewcommand{\theequation}{\arabic{section}-\arabic{equation}}

\def\wt{\widetilde}
\def\ol{\overline}
\def\ra{\rightarrow}
\def\lra{\longrightarrow}

\def\({$($}
\def\){$)$}
\def\chit{\chi_{\rm top}}
\def\bbp{\mathbb P}

\def\Pic{\text{{\rm Pic\,}}}
\def\rank{\text{{\rm rank\,}}}

\begin{document}

\title[On the slope of hyperelliptic fibrations]{On the slope of hyperelliptic fibrations with positive relative irregularity}

\author{Xin Lu}
\address{Department of Mathematics, East China Normal University, Shanghai, China, 200241}
\curraddr{Institut f\"ur Mathematik, Universit\"at Mainz, Mainz, Germany, 55099}
\email{lvxinwillv@gmail.com}
\thanks{This work is supported by SFB/Transregio 45 Periods, Moduli Spaces and Arithmetic of Algebraic Varieties of the DFG (Deutsche Forschungsgemeinschaft),
and partially supported by National Key Basic Research Program of China (Grant No. 2013CB834202).}

\author{Kang Zuo}
\address{Institut f\"ur Mathematik, Universit\"at Mainz, Mainz, Germany, 55099}
\email{zuok@uni-mainz.de}


\date{January, 29, 2015}



\phantomsection
\begin{abstract}
\addcontentsline{toc}{section}{Abstract}
Let $f:\,S \to B$ be a locally non-trivial relatively minimal fibration
of hyperelliptic curves of genus $g\geq 2$
with relative irregularity $q_f$.
We show a sharp lower bound on the slope $\lambda_f$ of $f$.
As a consequence, we prove
a conjecture of Barja and Stoppino on the lower bound of $\lambda_f$
as an increasing function of $q_f$ in this case,
and we also prove a conjecture of Xiao
on the ampleness of the direct image of the relative canonical sheaf if $\lambda_f<4$.
\end{abstract}

\maketitle


\section{Introduction}\label{introduction}
Let $f:\,S\to B$ be a fibration (or a family) of curves of genus $g\geq2$,
i.e., $f$ is a proper surjective morphism from a smooth complete surface $S$ to a smooth complete curve $B$ with connected fibers over complex number,
and the general fiber is a smooth complete curve of genus $g$.
If the general fiber is a hyperelliptic curve,
then we call $f$ a {\it hyperelliptic fibration}.
The fibration $f$ is called {\it relatively minimal},
if there is no $(-1)$-curve contained in the fibers of $f$.
Here a curve $C$ is called a {\it $(-k)$-curve}
if it is a smooth rational curve with self-intersection $C^2=-k$.
Without other statements, we always assume that fibrations in this paper
are relatively minimal.
The fibration $f$ is called {\it smooth} if all its fibers are smooth,
{\it isotrivial} if all its smooth fibers are isomorphic to each other,
{\it locally trivial} if it is both smooth and isotrivial,
and {\it semi-stable} if all its singular fibers are semi-stable.
Here a singular fiber $F$ of
$f$ is called {\it semi-stable} if it is a reduced nodal curve.

Let $\omega_S$ (resp. $K_S$) be the canonical sheaf (resp. the canonical divisor) of $S$.
Denote by $\omega_{S/B}=\omega_S\otimes f^*\omega_B^{\vee}$ (resp. $K_f=K_{S/B}=K_S-f^*K_B$) the
relative canonical sheaf (resp. the relative canonical divisor) of $f$.
If $f$ is relatively minimal, $K_f$ is numerical effective (nef),
i.e., $K_f\cdot C \geq 0$ for any curve $C\subseteq S$.
Set $b=g(B)$, $p_g=h^0(S,\,\omega_S)$, $q=h^1(S,\,\omega_S)$,
$\chi(\mathcal O_S)=p_g-q+1$, and let $\chit(S)$ be the topological Euler characteristic of $S$.
We consider the following relative invariants of $f$:
$$\begin{aligned}
\chi_f&=\deg f_*\omega_{S/B}=\chi(\mathcal O_S)-(g-1)(b-1),\\
K_f^2&=\omega_{S/B}\cdot \omega_{S/B}=K_S^2-8(g-1)(b-1),\\
e_f&=\chit(S)-4(g-1)(b-1),
\end{aligned}$$
They satisfy the Noether's formula:
\begin{equation}\label{formulanoether}
12\chi_f=K_f^2+e_f.
\end{equation}
If $f$ is relatively minimal, then these invariants are nonnegative,
and $\chi_f=0$ (equivalently, $K_f^2=0$)
if and only if $f$ is locally trivial (see \cite{arakelov}).
Note also that $e_f=0$ iff $f$ is smooth.

The {\it relative irregularity} $q_f$ of $f$ is defined to be
$$q_f=q-b.$$
It is clear that $0\leq q_f \leq g$.
The equality $q_f=g$ holds if and only if $S$ is birational to $B\times F$ by \cite{beaville82}.
And for $b\geq 1$, $q_f=0$ if and only if $f$ is the Albanese map of $S$.

If $f$ is not locally trivial, the {\it slope} of $f$ is defined to be
$$\lambda_f=\frac{K_f^2}{\chi_f}.$$
It follows immediately that $0< \lambda_f\leq 12$.
It turns out that the slope of a fibration is sensible to a lot of geometric
properties, both of the fibers of $f$ and of the surface $S$ itself (cf. \cite{ak00}).
We are mainly concerned with a lower bound of the slope. The main known
result in this direction is the slope inequality:
\begin{equation}\label{slopeinequality}
\text{If $g\geq 2$ and $f$ is not locally trivial, then $\lambda_f\geq \frac{4(g-1)}{g}$.}
\end{equation}
It was first proven by Horikawa
and Persson for hyperelliptic fibrations. Xiao gave a proof for general fibrations
(cf. \cite{xiao87}), and independently, Cornalba and Harris proved it for semi-stable fibrations
(cf. \cite{ch88}).

We would like to pay attention to the influence of the relative irregularity $q_f$
on the slope $\lambda_f$ of $f$.
It seems that the lower bound of $\lambda_f$ should be an increasing function of $q_f$.
The main influence of $q_f$ is the following Fujita decomposition
(cf. \cite{fujita78,fujita78b,kollar87}, see also Catanese-Dettweiler's recent paper \cite{cd14} for more detailed proof of the decomposition and discussion on the unitary factor):
\begin{equation}\label{fujitadecom}
f_*\omega_{S/B}=\mathcal A \oplus \mathcal F \oplus \mathcal O_B^{\oplus q_f},
\end{equation}
with $\mathcal A$ ample, $\mathcal F$ unitary and $\dim H^1(B, \Omega^1_B(\mathcal F))=0$.
The first result in this direction is due to Xiao (\cite{xiao87}):
\begin{equation*}
\text{If $q_f > 0$, then $\lambda_f\geq 4$ and the equality holds only if $q_f=1$.}
\end{equation*}
In particular, $q_f=0$ if $\lambda_f<4$.
He made the following conjecture (\cite[Conjecture 2]{xiao87}):
\begin{conjecture}[Xiao]\label{conjofxiao}
For any locally non-trivial fibration $f$,
$f_*\omega_{S/B}$ has no locally free quotient of degree zero
(i.e., $f_*\omega_{S/B}$ is ample) if $\lambda_f<4$.
\end{conjecture}
The above conjecture was confirmed to be true by Barja and Zucconi (cf. \cite{bz00})
under the assumption that $f$ is non-hyperelliptic or $g$ (or $b$) is small.

Related to the influence of $q_f$ on the lower bound of $\lambda_f$,
Xiao asked the following question (\cite[Problem 4]{xiao88}):
\begin{problem}[Xiao]\label{problemxiao}
Let $f:\,S \to B$ be a fibration of genus $g\geq 2$,
which is not locally trivial.
Find a good relationship between $\lambda_f$, $q_f$ and $g$.
\end{problem}

After that, there are lots of important results in this direction,
see for instance, \cite{bs08,bz01,cs08,konno93,konno94,luzuo14}.
Some explicit lower bounds depending on $q_f$ are also given in these literatures.
Among these, we would like to highlight the recent result obtained by Barja and Stoppino in \cite[Theorem\,1.3]{bs08}.
They proved that
\begin{equation}\label{bsresult}
\lambda_f\geq \frac{4(g-1)}{g-[m/2]},
\end{equation}
where $m:=\min\{{\rm Cliff}(f),\,q_f\}$ and ${\rm Cliff}(f)$ is defined to be the Clifford index of the general fiber of $f$.
When ${\rm Cliff}(f)$ is big enough, we see from \eqref{bsresult} that the lower bound of $\lambda_f$ is indeed an increasing function of $q_f$.
In fact, based on a lot of evidences, they made the following conjecture (\cite[Conjecture 1.1]{bs08}).
\begin{conjecture}[Barja-Stoppino]\label{conjecturebs}
Let $f:\,S \to B$ be as in Problem {\rm \ref{problemxiao}}. If $q_f < g-1$, then
\begin{equation}\label{conjectureequ}
\lambda_f\geq \frac{4(g-1)}{g-q_f}.
\end{equation}
\end{conjecture}

The bound \eqref{bsresult} obtained by Barja and Stoppino is very close to the conjectured bound \eqref{conjectureequ} once ${\rm Cliff}(f)$ is big.
In \cite{cs08} it was proved a bound for fibrations which are double covers of fibrations of genus $\gamma$.
This bound is analogous to the conjectured bound \eqref{conjectureequ}, but this second bound is not proven there.
In \cite[Example\,4.1]{bs08} it was proved that $\gamma=q_f$ and the two bounds coincide for double covers of trivial fibrations whose associated line bundle is ample.
In \cite{luzuo14}, we proved \eqref{conjectureequ} for semi-stable hyperelliptic fibrations.
We remark that if $q_f=g-1$, \eqref{conjectureequ} is known to be false (cf. \cite{bs08,pirola92}).

We are mainly interested in the lower bound of the slope of hyperelliptic fibrations, 
especially those with positive relative irregularity.
The main result is the following.
\begin{theorem}\label{maintheorem}
Let $f:\,S \to B$ be a locally non-trivial fibration
of hyperelliptic curves of genus $g\geq 2$ with relative irregularity $q_f$.
Then $q_f\leq \frac{g+1}{2}$, and
\begin{equation}\label{mainequation}
\lambda_f \geq \lambda_{g,q_f},
\end{equation}
where
\begin{equation}\label{defintionoflambda_gq_f}
\lambda_{g,q_f}=\left\{
\begin{aligned}
&8-\frac{4(g+1)}{(q_f+1)(g-q_f)},&\quad&\text{if~} q_f\leq \frac{g-1}{2};\\
&\frac{8(g-1)}{g},&&\text{if~} g \text{~is even, and~}q_f= \frac{g}{2};\\
&8,&&\text{if~}  g \text{~is odd, and~} q_f= \frac{g+1}{2}.
\end{aligned}\right.
\end{equation}
\end{theorem}

We will present examples to show that the bound \eqref{mainequation} is sharp.
It is not difficult to show that
$\lambda_{g,q_f}\geq \frac{4(g-1)}{g-q_f}.$
Therefore, we obtain
\begin{corollary}\label{corollarypfofconjectue}
For a locally non-trivial hyperelliptic fibration $f$, Conjecture {\rm\ref{conjecturebs}} is true.
Moreover, the inequality \eqref{conjectureequ} can become an equality only if $q_f=0$, $\frac{g-1}{2}$,
$\frac{g}{2}$ or $\frac{g+1}{2}$.
In particular, $g\leq 3$ if $q_f=1$ and $\lambda_f=4$.
\end{corollary}

Note that a fibration $f$ is hyperelliptic if and only if ${\rm Cliff}(f)=0$.
Hence our result is somewhat orthogonal to the cases treated in \cite{bs08}.
In particular, this seems to indicate that the Clifford index has no essential influence in the relation between the slope and the relative irregularity.

According to \cite[Theorem\,4.7]{luzuo14} (see also Theorem \ref{thmFtrivial}), if $f$ is hyperelliptic, then $\mathcal F=0$ (i.e., there is no non-trivial unitary part)
in \eqref{fujitadecom} after a suitable finite \'etale base change.
Hence by our theorem, Conjecture \ref{conjofxiao} is true when $f$ is hyperelliptic.
Combining with the result of Barja and Zucconi (cf. \cite{bz00}), we prove
\begin{corollary}\label{corollarypfofconjectuexiao}
Conjecture {\rm \ref{conjofxiao}} is true.
\end{corollary}

Our paper is organized as follows.
In Section \ref{preliminaries}, we review some basic properties about
a hyperelliptic fibration $f:\, S \to B$  mainly due to Xiao Gang.
By blowing up the isolated fixed points of the hyperelliptic involution, we get a double cover
$\pi:\, \wt S \to \wt P$ of smooth projective surfaces.
We then define the local relative invariants $s_i$ for $2\leq i \leq g+2$,
and show in Theorem \ref{xiaotheorem} that the global relative invariants of $f$ can be expressed by
those local invariants.
In Section \ref{sectionpositive},
we restrict ourselves to the case where the relative irregularity is positive,
and prove an inequality \eqref{s2leqDelta} involving these invariants $s_i$'s.
The proof starts from the observation
that the double cover $\tilde\pi:\,\wt S \to \wt P$ is fibred.
In Section \ref{proofofmain}, we prove Theorem \ref{maintheorem} and its corollaries.
When $q_f=0$, \eqref{mainequation} is nothing new but \eqref{slopeinequality}.
If $q_f>0$, \eqref{mainequation} follows from \eqref{s2leqDelta} and
the formulas given in Theorem \ref{xiaotheorem}.
In Section \ref{examples}, we present examples
to show that the bound \eqref{mainequation} is sharp.
Finally in the Appendix, we provide a proof of a theorem on the unitary part of a semi-stable hyperelliptic fibration, which is used in the proof of Corollary \ref{corollarypfofconjectuexiao}.

\section{Preliminaries}\label{preliminaries}
\subsection{Double covers}
In this subsection, we review some basic properties of double covers
(cf. \cite[\S\,V.22]{bhpv} and \cite[\S\,2]{xiao91}).

A double cover $\pi:\,X\to Y$ of a smooth projective surface $Y$ is determined by
a line bundle $L$ over Y and a section $s\in H^0(Y,L^2)$ where
$$X={\rm Proj}\left(\bigoplus_{i=0}^{+\infty} L^{ -i}\right)\bigg/\Big\langle s^{\vee}\Big\rangle.$$

Let $R$ be the zero divisor of $s$. Then $L^{\otimes2}\cong \mathcal O_Y(R)$, the map $\pi$ is completely determined by the pair $(R,L)$, and
$X$ is smooth if and only if $R$ is smooth.
%
To obtain a smooth double cover from a double cover $\pi: X\to Y$ of a smooth projective surface $Y$,
we perform the {\it canonical resolution} (cf. \cite[\S\,III.7]{bhpv}).
\begin{center}\mbox{}
  \xymatrix{
\wt X\ar@{=}[r] & X_{t} \ar[r]^-{\phi_t}\ar[d]_-{\tilde \pi=\pi_t}&
 X_{t-1}\ar[r]^-{\phi_{t-1}}\ar[d]^-{\pi_{t-1}}&
\cdots \ar[r]^-{\phi_2} & X_1\ar[r]^-{\phi_1}\ar[d]_-{\pi_1}
& X_0 \ar[d]^-{\pi_0=\pi}\ar@{=}[r] & X\\
\wt Y\ar@{=}[r] & Y_{t} \ar[r]^-{\psi_t}& Y_{t-1}\ar[r]^-{\psi_{t-1}}&
\cdots \ar[r]^-{\psi_2} & Y_1\ar[r]^-{\psi_1} & Y_0 \ar@{=}[r] & Y}
\end{center}
where $\wt X=X_t$ is smooth and $\psi_i$'s are successive blowing-ups
resolving the singularities of $R$;
$\pi_{i}:\,X_i \to Y_i$ is the double cover determined by
$(R_i,\,L_i)$ with
\begin{equation}\label{eqncomupteinvdouble}
R_i=\psi_i^*(R_{i-1})-2[m_{i-1}/2]\, \mathcal E_i,\qquad
L_i=\psi_i^*(L_{i-1})\otimes \mathcal O_{Y_i}\left(\mathcal E_i^{-[m_{i-1}/2]}\right),
\end{equation}
where $\mathcal E_i$ is the exceptional divisor of $\psi_i$,
$m_{i-1}$ is the multiplicity of the singular point $y_{i-1}$ in $R_{i-1}$,
$[~]$ stands for the integral part,
$R_0=R$ and $L_0=L$.

The morphism
$$\psi\triangleq \psi_1\circ\cdots\circ\psi_t:~\wt Y \lra Y$$
is also called a {\it minimal even resolution} of $R$.
We call a singularity $y_j \in R_{j}\subseteq Y_{j}$ infinitely close to
$y_{i-1}\in R_{i-1}\subseteq Y_{i-1}$ ($j\geq i$), if
$\psi_i\circ\cdots\circ\psi_j(y_j)=y_{i-1}\,.$

\begin{definition}\label{defofniglig}
A singularity $y_{i-1}\in R_{i-1}\subseteq Y_{i-1}$ above is said to be negligible if $[m_{i-1}/2]=1$,
and $[m_j/2]\leq 1$ for any $y_j \in R_{j}\subseteq Y_{j}$ ($j\geq i$)
infinitely close to $y_{i-1}$.
In this case, the blowing-up $\psi_{i}:\,Y_{i}\to Y_{i-1}$ is called a negligible blowing-up.
\end{definition}

It is easy to see that $\psi$ can be decomposed into $\tilde\psi:\, \wt Y \to \hat Y$
and $\hat\psi:\,\hat Y \to Y$, where $\tilde\psi$ and $\hat\psi$ are composed
of negligible and non-negligible blowing-ups respectively.
We call $\hat\psi$
the {\it minimal even resolution of non-negligible singularities} of $R$.


\subsection{Invariants of hyperelliptic fibrations}
In this subsection, we review some results about hyperelliptic fibrations
mainly due to Xiao (cf. \cite[\S\,2]{xiao91} \& \cite[\S\,5.1]{xiao92}).

Let $f:\,S \to B$ be a relatively minimal hyperelliptic fibration.
The relative canonical map of $f$ is generically of degree $2$. This map determines an
involution $\sigma$ on S whose restriction on a general fiber $F$ of $f$ is the hyperelliptic involution
of $F$. The involution $\sigma$ is called the hyperelliptic involution associated to $f$.

Let $\vartheta: \wt S \to S$ be the composition of all the blowing-ups of isolated fixed points of the
hyperelliptic involution, and let $\tilde\sigma$ be the induced involution on $\wt S$.
The quotient space $\wt P = \wt S/\langle\tilde\sigma\rangle$ is
a smooth surface, and $f$ induces a ruling $\tilde h:\,\wt P \ra B$ on $\wt P$.
The quotient map $\tilde\pi:\,\wt S\to\wt P$ is a double cover which is determined by
the pair $(\wt R,\,\wt L)$, where $\wt R$ is the branch locus of $\tilde\pi$
and $\wt L$ is a line bundle such that $\mathcal O_{\wt P}\big(\wt R\big) \cong \wt L^{\otimes 2}$.

For any contraction $\varphi:\,\wt P \to P'$ and $R'=\varphi(\wt R)$, the double cover $\tilde \pi$ induces a double cover
$S'\to P'$, which is assumed to be determined by the pair $(R',\,L')$. For convenience, we simply call $(R',\,L')$ the image of $(\wt R,\,\wt L)$.
\begin{lemma}[\cite{xiao91,xiao92}]\label{contractionbyxiao}
There exists a contraction of ruled surfaces
$\psi:\,\wt P \to P$:
$$\xymatrix{
  \wt P \ar[rr]^-{\psi} \ar[dr]_-{\tilde h}
                &  &    P  \ar[dl]^-{ h}   \\
                & B                 }$$
such that
$P$ is a geometrical ruled surface (i.e., any fiber of $ h$ is $\bbp^1$),
the singularities of $R$ are at most of multiplicity $g+2$,
and the self-intersection $R^2$ is the smallest among all such choices,
where $(R,L)$ is the image of $(\wt R,\wt L)$ in $P$.
\end{lemma}

Note that $\psi$ can be decomposed into $\tilde\psi:\,\wt P\to \hat P$ and $\hat\psi:\,\hat P \to P$
in the following diagram,
where $\hat\psi:\, \hat P \to P$ is a
minimal even resolution of non-negligible singularities of $R$.
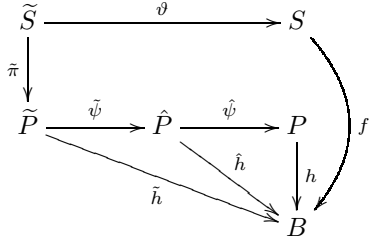
\begin{figure}[H]
\begin{center}\mbox{}
  \xymatrix@C=1.3cm{
     \wt S \ar@{->}[d]_-{\tilde\pi}\ar@{->}[rr]^-{\vartheta} &   & S \ar@/^7mm/"3,3"^f\\
      \wt P \ar@{->}[r]^-{\tilde\psi} \ar@{->}[drr]_-{\tilde h}
           & \hat{P}\ar@{->}[r]^-{\hat\psi} \ar@{->}[dr]^-{\hat h} &
           P \ar@{->}[d]^-{ h}\\
           &&B}
\end{center}
\caption{Hyperelliptic fibration.\label{hyperellipticdiagram}}\vspace{-0.4cm}
\end{figure}

Let $(\hat R,\,\hat L)$ be the
image of $(\wt R,\, \wt L)$ in $\hat P$.
Let $\hat\psi=\hat\psi_1\circ\cdots\circ\hat\psi_t$ be the decomposition of $\hat\psi$,
where $\hat\psi_i: \hat P_i \to \hat P_{i-1}$ is a blowing-up at $y_{i-1}$,
$\hat P_0=P$ and $\hat P_t=\hat P$.
Let $\hat R_i$ be the image of $\hat R$ in $\hat P_i$.
It could happen that there is one or more singular points of $\hat R_i$
over the exceptional curve $\hat{\mathcal E_i}$ of $\hat\psi_i$.
We remark that the decomposition of $\hat\psi$ is not unique.
If $y_{i-1}$ is a singular point of $\hat R_{i-1}$ of odd multiplicity $2k+1$ ($k\geq 1$)
and there is a unique singular point $y$ of $\hat R_i$
on the exceptional curve $\hat{\mathcal E_i}$ of multiplicity $2k+2$,
then we always assume that $\hat\psi_{i+1}: \hat P_{i+1} \to \hat P_{i}$ is the standard blowing-up at $y_i=y$.
We call such a pair $(y_{i-1},y_{i})$ a {\it singularity of $R$ of type $(2k+1 \to 2k+1)$},
and call $y_{i-1}$ (resp. $y_i$) the first (resp. second) component of such a singularity.


\begin{definition}\label{definitionofs_i}
For any singular fiber $F$ of $f$ and $3\leq i\leq g+2$,
the $i$-th singularity index of $F$ is defined as follows
(with respect to the contraction $\psi$):
\begin{list}{}
{\setlength{\labelwidth}{0.3cm}
\setlength{\leftmargin}{0.4cm}}
\item[$\bullet$] if $i$ is odd, $s_i(F)$ equals the number of $(i\to i)$
                 type singularities of $R$ over the image $f(F)$;
\item[$\bullet$] if $i$ is even, $s_i(F)$ equals the number of singularities of multiplicity $i$ or $i+1$ of $R$ over the image $f(F)$,
                 neither belonging to the second component of $(i-1 \to i-1)$ type singularities nor to the first component of $(i+1 \to i+1)$  type singularities.
\end{list}
We remark that the infinitely close singularities of $R$ should also be taken into consideration when defining $s_i(F)$ above for even $i$.
Let $K_{\hat P/B}=K_{\hat P}-\hat h^*K_{B}$ and $R'=\hat R \setminus\hat V$,
where $\hat V$ is the union of isolated vertical $(-2)$-curves in $\hat R$.
Here a curve $C\subseteq \hat R$ is called to be {\it isolated} in $\hat R$,
if there is no other curve $C'\subseteq \hat R$ such that $C\cap C' \neq \emptyset$. We define
$$s_2\,\triangleq(K_{\hat P/B}+R')\cdot R',\quad\text{~and~}\quad s_i\,\triangleq\sum_{F \text{~is singular}} s_i(F),\quad 3\leq i\leq g+2.$$
By definition, $s_i$ is non-negative for $i\geq 3$, but it is not clear whether $s_2$ is non-negative or not.
\end{definition}

\begin{lemma}[\cite{xiao91,xiao92}]
These singularity indices $s_i$'s defined above are independent on the choices of $\psi$ in Lemma {\rm\ref{contractionbyxiao}}.
\end{lemma}
\noindent
We just remark that the independence of $s_2$ on $\psi$ is contained implicitly in \cite[Lemma\,8]{xiao91}, which proves the independence of $\tilde\psi$ and $s_i$ on $\psi$ for $i\geq 3$.

\begin{remark}
In \cite{liutan14}, it is proven that if $f$ is semi-stable, then
$s_{g+2}=0$ and
$s_2+2\sum\limits_{k=2}^{[g/2]} s_{2k}$
\big(resp. $s_{2k+1}$ for $1\leq k \leq [g/2]$\big) is the number of nodes of type $0$ (resp. $k$) contained in the fibers of $f$.
\end{remark}

\begin{theorem}[\cite{xiao91,xiao92}]\label{xiaotheorem}
Let $f:\,S\to B$ be a fibration of hyperelliptic curves of genus $g\geq2$,
and $s_i$'s the singularity indices as above. Then
$$\begin{aligned}
(2g+1)K_f^2&=(g-1)\Big(s_2+(3g+1)s_{g+2}\Big)+\sum_{k=1}^{[\frac{g}{2}]}
a_ks_{2k+1}+\sum_{k=2}^{[\frac{g+1}{2}]}b_ks_{2k},\\
(2g+1)\chi_f&=\frac{gs_2+(g^2-2g-1)s_{g+2}}{4}
+\sum_{k=1}^{[\frac{g}{2}]} k(g-k)s_{2k+1}+\sum_{k=2}^{[\frac{g+1}{2}]}\frac{k(g-k+1)}{2}s_{2k},\\
\end{aligned}$$
where
$a_k=12k(g-k)-2g-1$ and $b_k=6k(g-k+1)-4g-2.$
\end{theorem}
For the convenience, we reproduce a proof of Theorem \ref{xiaotheorem}.
To start it, we need
\begin{lemma}[\cite{xiao91,xiao92}]\label{numberofcontraction}
Let $F$ be a singular fiber of the fibration $f$,
and $\wt F$ (resp. $\hat \Gamma$)
the corresponding fiber in $\wt S$ (resp. $\hat P$).
Then the $(-1)$-curves in $\wt F$ are in one-to-one correspondence to the isolated $(-2)$-curves
of $\hat R$, which are also contained in $\hat \Gamma$.
And the number of these $(-1)$-curves is equal to
$$2s_{g+2}(F)+\sum_{k=1}^{[\frac{g}{2}]} s_{2k+1}(F).$$
\end{lemma}

\begin{proof}[Proof of Theorem {\rm\ref{xiaotheorem}}]
Let
\begin{equation}\label{n=delta^2/g+1}
n=\frac{L^2}{g+1}.
\end{equation}
Then $K_{P/B}\cdot L=-n$, where $K_{P/B}=K_P- h^*K_B$.
As $\hat\psi:\, \hat P \to P$ is a
minimal even resolution of non-negligible singularities of $R$,
by Definition \ref{definitionofs_i} and the formula \eqref{eqncomupteinvdouble}, one gets
(we should remark that the minimality of the self-intersection $R^2$ in Lemma \ref{contractionbyxiao} implies that $s_{g+2}=0$ if $g$ is even;
otherwise, there is a singularity $p$ of $R$ with multiplicity $g+2$.
In this case, an elementary transformation of $P$ centered at $p$ gives another contraction $\psi':\wt P \to P'$
with $\big(\psi'(\wt R)\big)^2< R^2$, which is a contradiction to the minimality of $R^2$.)
$$\begin{aligned}
\hat L^2&=L^2-\frac{g^2+4g+5}{2}s_{g+2}
-\sum_{k=1}^{[\frac{g}{2}]}(2k^2+2k+1)s_{2k+1}-\sum_{k=2}^{[\frac{g+1}{2}]}k^2s_{2k},\\
K_{\hat P/B}\cdot\hat L&=K_{P/B}\cdot L+(g+2)s_{g+2}+
\sum_{k=1}^{[\frac{g}{2}]}(2k+1)s_{2k+1}+\sum_{k=2}^{[\frac{g+1}{2}]}ks_{2k}.
\end{aligned}$$
So by the definition of $s_2$ and Lemma \ref{numberofcontraction}, one has
$$\begin{aligned}
&s_2-4s_{g+2}-2\sum\limits_{k=1}^{[\frac{g}{2}]} s_{2k+1}=(K_{\hat P/B}+\hat R)\cdot \hat R=(K_{\hat P/B}+2\hat L)\cdot 2\hat L\\
=~&(4g+2)n-2(g^2+3g+3)s_{g+2}
-\sum_{k=1}^{[\frac{g}{2}]}2(4k^2+2k+1)s_{2k+1}-\sum_{k=2}^{[\frac{g+1}{2}]}2k(2k-1)s_{2k}.
\end{aligned}$$
Hence
\begin{equation}\label{n=s_i}
n=\frac{1}{2(2g+1)}s_2+\frac{g^2+3g+1}{2g+1}s_{g+2}
+\sum_{k=1}^{[\frac{g}{2}]}\frac{4k^2+2k}{2g+1}s_{2k+1}+\sum_{k=2}^{[\frac{g+1}{2}]}\frac{2k^2-k}{2g+1}s_{2k}.
\end{equation}
Let $b=g(B)$ be the genus of $B$ and $K_{\hat P/B}=K_{\hat P}- \hat h^*K_B$.
Note that all singular points of $\hat R \subseteq \hat P$
are negligible (cf. Definition \ref{defofniglig})
by construction. According to the standard formulas for double covers (cf. \cite[\S\,V.22]{bhpv}), we get
$$\begin{aligned}
\chi_{\tilde f}&
=2\chi(\mathcal O_{\hat P})+\frac12(\hat L^2+K_{\hat P}\cdot \hat L)-(g-1)(b-1)
=\frac12(\hat L^2+K_{\hat P/B}\cdot \hat L)\\
&=\frac{gs_2+(g^2-2g-1)s_{g+2}}{4(2g+1)}
   +\sum_{k=1}^{[\frac{g}{2}]} \frac{k(g-k)}{2g+1}s_{2k+1}
   +\sum_{k=2}^{[\frac{g+1}{2}]}\frac{k(g-k+1)}{2(2g+1)}s_{2k},\\[0.2cm]
K_{\tilde f}^2&
=2(\hat L+K_{\hat P})^2-8(g-1)(b-1)
=2(\hat L+K_{\hat P/B})^2\\
&=\frac{g-1}{2g+1}s_2+\frac{3(g^2-2g-1)}{2g+1}s_{g+2}+\sum_{k=1}^{[\frac{g}{2}]}
\left(\frac{a_k}{2g+1}-1\right)s_{2k+1}+\sum_{k=2}^{[\frac{g+1}{2}]}\frac{b_k}{2g+1}s_{2k}.~\,\qquad
\end{aligned}$$
Note that $\chi_{f}=\chi_{\tilde f}$, and
$K_f^2=K_{\tilde f}^2+2s_{g+2}+\sum\limits_{k=1}^{[g/2]} s_{2k+1}$
by Lemma \ref{numberofcontraction}.
The theorem follows immediately.
\end{proof}

\section{Hyperelliptic fibrations with positive relative irregularity}\label{sectionpositive}
The purpose of this section is to prove the following inequality for
a locally non-trivial hyperelliptic fibration with positive relative irregularity.
\begin{proposition}\label{indeplemma}
Let $f:\,S \to B$ be a hyperelliptic fibration of genus $g$, which is not locally trivial.
Let $s_i$ $(2\leq i\leq g+2)$ be the $i$-th singularity index of $f$ defined in Definition {\rm \ref{definitionofs_i}}.
Assume that the relative irregularity $q_f>0$.
Then
\begin{equation}\label{s2leqDelta}\hspace{-3mm}
\begin{aligned}
&s_2+\sum_{k=1}^{q_f-1}4k(2k+1) s_{2k+1}+\sum_{k=2}^{q_f}2k(2k-1) s_{2k}\\
\leq~& \sum_{k=q_f}^{[\frac{g}{2}]}\frac{(2k+1)(2g+1-2k)}{g+1} s_{2k+1}
+\sum_{k=q_f+1}^{[\frac{g+1}{2}]}\frac{2k(g+1-k)}{g+1} s_{2k}+(g+1) s_{g+2}.
\end{aligned}
\end{equation}
\end{proposition}

In order to prove the above proposition,
we always assume that $q_f>0$ in the section.
Let $\tilde\pi:\,\wt S \to \wt P$ be the induced double cover
with branch divisor $\wt R \subseteq \wt P$ as in Figure\,\ref{hyperellipticdiagram}.
The strategy of the proof is as follows. The starting point is the observation that
the double cover $\tilde\pi$ is fibred (see \cite{khashin83} or \cite[Definition\,6.2]{luzuo14} for the definition).
Based on this observation, we prove that there exists a fibration $\tilde h':\,\wt P \to \mathbb P^1$ such that the branch divisor $\wt R$
of $\tilde\pi$ is contained in exactly $2q_f+2$ fibers of $\tilde h'$.
Let $\psi$ be the contraction in Lemma \ref{contractionbyxiao}. We decompose it into $\bar\psi:\wt P \to \overline P$
and $\check\psi:\overline P \to P$ such that there is an induced fibration $\bar h':\,\overline P \to \mathbb P^1$
and any $(-1)$-curve contracted by $\check \psi$ is not contracted by $\bar h'$.
Then Proposition \ref{indeplemma} follows from the facts that the contraction $\check \psi$
contributes only to $s_i$ with $i>2q_f$ and $\overline R=\bar\psi(\wt R)$ is semi-negative definite.

The fact that the double cover $\tilde\pi$ is fibred, is proved in \cite{luzuo14} under the extra assumption that $f:\,S \to B$ is semi-stable.
But the proof there does not use this assumption. Hence we have
\begin{proposition}[{\cite[Propositions 6.4 and 6.5]{luzuo14}}]\label{fibredprop}
The double cover $\tilde\pi:\,\wt S \to \wt P$ in Figure\,\ref{hyperellipticdiagram} is fibred, i.e.,
there exists a double cover $\pi': B' \to \bbp^1$ of smooth projective
curves and two morphisms $\tilde f':\, \wt S\to B' $ and $\tilde h':\, \wt P \to \bbp^1$, such that
the diagram
\begin{center}\mbox{}
\xymatrix{
 \wt S \ar@{->}[rr]^-{\tilde\pi}\ar@{->}[d]_-{\tilde f'} && \wt P \ar@{->}[d]^-{\tilde h'}\\
 B' \ar@{->}[rr]^-{\pi'}                  && \bbp^1
}
\end{center}
is commutative, $\wt R$ is contained in the fibers of $\tilde h'$ and
\begin{equation}\label{boundofq_f}
q_f=q(\wt S)-q(\wt P) = g(B')\leq \frac{g+1}{2}.
\end{equation}
\end{proposition}

%
%

\begin{remark}
Let $f:\,S\to B$ be as in the Proposition \ref{indeplemma} with $g(B)\geq 1$, and $d\geq 2$ the degree of the Albanese map $S\to {\rm Alb\,}(S)$.
Xiao (\cite{xiao92-0}) proved a more precise description on $q_f$:
$$\frac{g+1}{d}-1\leq q_f\leq \frac{g-1}{d}+1.$$
\end{remark}

Based on Proposition \ref{fibredprop}, we would like to show that the branch divisor $\wt R$
of $\tilde\pi:\,\wt S \to \wt P$ has a very special form which is described in Lemma \ref{wtR=Gamma_i}.

Let $\tilde f':\, \wt S \to B'$ be the fibration in
Proposition \ref{fibredprop}.
Since $g(B')=q_f\geq1$,
it follows that any $(-1)$-curve in $\wt S$ is contracted by $\tilde f'$.
Hence $\tilde f'$ factors through $\vartheta:\,\wt S \to S$.
Let $f':\,S \to B'$ be the induced map.
Note that the fibration $\tilde f':\, \wt S \to B'$ in Proposition \ref{fibredprop}
is clearly unique.
Hence the hyperelliptic involution $\sigma$ induces an involution $\sigma'$ on $B'$
such that $B'/\langle\sigma'\rangle\cong \bbp^1$ with the following diagram:
\begin{figure}[H]
\begin{center}\mbox{}
  \xymatrix@C=1.1cm{
     \wt S \ar@/^5mm/"1,3"^{\tilde f'} \ar@/_5mm/"3,1"_{\tilde f}
     \ar@{->}[d]^-{\tilde\pi}\ar@{->}[r]_-{\vartheta} &S\ar@{->}[r]_-{f'}& B'\ar@{->}[d]_-{\pi'}\\
      \wt P \ar@{->}[rr]^-{\tilde h'} \ar@{->}[d]^-{\tilde h}
           & &           \bbp^1\cong B'/\langle\sigma'\rangle\\
           B&&}
\end{center}
\caption{Hyperelliptic fibration with positive relative irregularity.\label{bimapdiagram}}\vspace{-0.4cm}
\end{figure}

Assume $\pi':\, B' \to \bbp^1$ is branched over $\Delta\subseteq \bbp^1$.
Applying Hurwitz formula to the double cover $\pi'$, one sees that $|\Delta|=2q_f+2$.
For any $y\in \Delta$, let $\wt\Gamma_y'=\sum \tilde n_C' C$ be the fiber of $\tilde h'$ over $y$,
and
$$\wt \Gamma_{y,{\rm o}}'\triangleq
\sum_{C\subseteq \wt \Gamma_y' \atop \tilde n_C'\text{~is odd}}  C\subseteq \wt \Gamma_y'.$$
According to Proposition \ref{fibredprop}, $\wt R$ is contained in the fibers of $\tilde h'$.
In fact, we can prove an explicit expression of $\wt R$.

\begin{lemma}\label{wtR=Gamma_i}
$$\wt R=\sum_{y\in\Delta} \wt \Gamma_{y,{\rm o}}'.$$
\end{lemma}
\begin{proof}
Let $B'\times_{\bbp^1}\wt P$ be the fiber-product, and $X\to B'\times_{\bbp^1}\wt P$ the normalization.
By the universal property of the fiber-product (cf. \cite[\S\,II-2]{hartshorne}),
there exists a unique morphism
$\gamma':\, \wt S \to B'\times_{\bbp^1}\wt P$.
Since $\wt S$ is smooth, there also exists a unique morphism $\gamma:\,\wt S \to X$,
such that the following diagram is commutative.
\begin{center}\mbox{}
\xymatrix@C=1.1cm{
  \wt S \ar@/_1cm/"3,2"_-{\tilde f'} \ar@{->}[dr]_-{\exists \,!~\gamma}
  \ar@{->}[drr]^-{\exists \,!~\gamma'} \ar@/^7mm/"2,5"^-{\tilde\pi} &&&&\\
  &X\ar@{->}[r]^-{\pi_2}\ar@{->}[d]& B'\times_{\bbp^1}\wt P \ar@{->}[dl]\ar@{->}[rr]^-{\pi_1}
  &&\wt P \ar@{->}[d]^-{\tilde h'}\\
  &B'\ar@{->}[rrr]^-{\pi'}&&&\bbp^1
}
\end{center}
Clearly the composition $\pi_1\circ\pi_2:\,X\to \wt P$ is a double cover branched exactly over
$$\sum_{y\in\Delta} \wt \Gamma_{y,{\rm o}}'.$$
Therefore, it suffices to prove that $\gamma$ is an isomorphism.

As $\deg \tilde\pi=\deg (\pi_1\circ\pi_2)$, we get $\deg \gamma=1$, i.e.,
$\gamma:\, \wt S\to X$ is a contraction of curves.
Note that $\tilde\pi$ does not contract any curve.
Neither does $\gamma$ because any curve contracted by $\gamma$ must be also
contracted by $\tilde\pi$.
This completes the proof.
\end{proof}

The contraction $\psi:\, \wt P \to P$ in Lemma \ref{contractionbyxiao} is composed of several blowing-ups.
We divide those blowing-ups as $\psi=\check\psi\circ\bar\psi$,
where $\bar\psi:\, \wt P \to \ol P$ is the largest contraction such that
$\tilde h'$ factors through $\bar\psi$.
So we have the following diagram:
\begin{figure}[H]
\begin{center}\mbox{}
  \xymatrix@=1cm{
     &\bbp^1\cong B'/\langle\sigma'\rangle &\\
      \wt P \ar@{->}[r]^-{\quad\bar\psi} \ar@{->}[dr]_{\tilde h} \ar@{->}[ru]^-{\tilde h'}
 & \ol P\ar@{->}[r]^-{\check\psi}\ar@{->}[d]^-{\bar h}\ar@{->}[u]_-{\bar h'}& P \ar@{->}[dl]^-{h}\\
           &B&}
\end{center}
\caption{Decomposition of $\psi$.\label{bimapfactordiagram}}\vspace{-0.4cm}
\end{figure}

Next we want to show that the contraction $\check \psi$ contributes only to $s_i$ with $i>2q_f$. For this purpose, we first prove
%
%

\begin{lemma}\label{ERgeq2qf+2}
Let $\mathcal E\subseteq \ol P$ be any $(-1)$-curve contracted by $\check \psi$. Then $\mathcal E\cdot \ol R$ is even and
$$\mathcal E\cdot \ol R \geq 2(q_f+1).$$
\end{lemma}
\begin{proof}
Let $(\ol R,\,\ol L;\,\ol\Gamma_y')$ be the image of $(\wt R,\,\wt L;\,\wt\Gamma_y')$ on $\ol P$, where $y\in\Delta$.
Let $\ol m=\mathcal E\cdot \ol R$, $\sigma:\, \ol P \to \ol P_1$ the contraction of $\mathcal E$,
$x$ the image of $\mathcal E$, and $(\ol R_1,\,\bar L_1)$ the image of $(\ol R,\,\bar L)$ on $\ol P_1$.
Then $x$ is a singularity of $\ol R_1$ of multiplicity $\ol m$, and $\mathcal E$ is mapped surjectively onto $\mathbb P^1$ by $\bar h'$ according to the construction of $\bar h'$ in Figure \ref{bimapfactordiagram}.
Let
$$
\ol\Gamma_{y,{\rm o}}' ={\sum_{C\subseteq \ol \Gamma_y' \atop \bar n_C'\text{~is odd}} C}
\text{\quad and\quad}
\ol \Gamma_{y,{\rm r}}'=\sum_{C\subseteq \ol \Gamma_y' \atop \bar n_C'=1} C,\qquad
\text{if~~~}\ol\Gamma_y'=\sum \bar n_C' C.$$
By Lemma\,\ref{wtR=Gamma_i},
$$\ol R=\bar\psi(\wt R)=\sum\limits_{y\in\Delta}\ol\Gamma_{y,{\rm o}}'$$
is contained in the fibers of $\bar h'$.
Hence
$\mathcal E \nsubseteq \ol R$,
from which it follows that $\ol m=\mathcal E\cdot \ol R$ is even.
Let
$$\ol R_{\rm all}=\sum_{y\in\Delta} \ol \Gamma_{y}'=(\bar h')^*(\Delta) \text{\quad and \quad}
\ol R_{\rm r}=\sum_{y\in\Delta} \ol \Gamma_{y,{\rm r}}'.$$
Then $\ol R_{\rm r}\subseteq \ol R\subseteq \ol R_{\rm all}$.
To complete the proof, it is enough to prove that
$$\mathcal E\cdot \ol R_{\rm r}\geq 2(q_f+1).$$

Note that the restricted morphism $\bar h'|_{\mathcal E}:\,\mathcal E \to \bbp^1$ is surjective.
For any $p\in \mathcal E\cap \ol R_{\rm all}$,
let $r_p=I_p(\mathcal E,\,\ol R_{\rm all})$ be the local intersection number.
Since $\ol R_{\rm all}=(\bar h')^*(\Delta)$ consists of $|\Delta|=2q_f+2$ fibers of $\bar h'$, one has
$$\sum_{p\in \mathcal E\cap \ol R_{\rm all}}
r_p=\mathcal E\cdot \ol R_{\rm all}=\deg(\bar h'|_{\mathcal E})\cdot (2q_f+2).$$
By definition, $r_p\geq 2$ for any
$p\in \left(\mathcal E\cap \ol R_{\rm all}\right) \setminus \left(\mathcal E\cap \ol R_{\rm r}\right)$.
On the other hand, as $\mathcal E$ is a $(-1)$-curve, the ramification number of
$\bar h'|_{\mathcal E}$ is $2\deg(\bar h'|_{\mathcal E})-2$.
So
\begin{eqnarray*}
2\deg(\bar h'|_{\mathcal E})-2 &\geq& \sum_{p\in \mathcal E\cap \ol R_{\rm all}} (r_p-1)
=\sum_{p\in \left(\mathcal E\cap \ol R_{\rm all}\right) \setminus \left(\mathcal E\cap \ol R_{\rm r}\right)} (r_p-1)
~+\sum_{p\in \mathcal E\cap \ol R_{\rm r}} (r_p-1)\\[.15cm]
&\geq&
\sum_{p\in \left(\mathcal E\cap \ol R_{\rm all}\right) \setminus \left(\mathcal E\cap \ol R_{\rm r}\right)} \frac{r_p}2
~+\sum_{p\in \mathcal E\cap \ol R_{\rm r}} \frac{(r_p -1)}2\\[.15cm]
&=& \frac12\sum_{p\in \mathcal E\cap \ol R_{\rm all}} r_p - \frac{\big|{\mathcal E\cap \ol R_{\rm r}}\big|}2
=\deg(\bar h'|_{\mathcal E})\cdot (q_f+1)-\frac{\big|{\mathcal E\cap \ol R_{\rm r}}\big|}2.
\end{eqnarray*}
Therefore
$$\mathcal E\cdot \ol R_{\rm r}\geq \big|{\mathcal E\cap \ol R_{\rm r}}\big|
\geq 2\deg(\bar h'|_{\mathcal E})\cdot(q_f-1)+4 \geq 2(q_f-1)+4=2(q_f+1).$$
The proof is complete.
\end{proof}

We assume that $\check\psi=\check\psi_1\circ\cdots\circ\check\psi_u$,
where $\check\psi_i: \check P_i \to \check P_{i-1}$ is a blowing-up at $\check x_{i-1}\in \check P_{i-1}$
with exceptional curve $\check {\mathcal E}_{i}\subseteq \check P_{i}$,
$\check P_0=P$ and $\check P_u=\ol P$.
Let $\check R_i$ be the image of $\ol R$ in $\check P_i$,
and $\check x_i$ a singularity of $\check R_i$ of multiplicity $\check m_i$.

\begin{lemma}\label{m_igeq2q_f+1}
For $1\leq i\leq u-1$, we have either $\check m_i \geq 2(q_f+1)$,
or $\check m_i=2q_f+1$ and $\check x_{i+1}$ is the unique singular point on
$\check {\mathcal E}_{i+1} \subseteq \check P_{i+1}$ of multiplicity $\check m_{i+1}=2q_f+2$.
\end{lemma}
\begin{proof}
First, we show that $\check m_i \geq 2q_f+1$ for any $1\leq i\leq u-1$.
Since $\check \psi$ is a part of the even resolution of $R=\check R_0$, one has
\begin{equation}\label{m_im_i+1}
\left\{\begin{aligned}
&\text{if $\check m_i$ is even, then $\check{\mathcal E}_{i+1}\nsubseteq \check R_{i+1}$,
and so $\check m_{i+1} \leq \check m_i$;}\\
&\text{if $\check m_i$ is odd, then $\check{\mathcal E}_{i+1}\subseteq \check R_{i+1}$,
and so $\check m_{i+1} \leq \check m_i+1$.}
\end{aligned}\right.
\end{equation}
By induction, for any singularity $\check x_{i+j}$, infinitely close to $\check x_i$,
we have $\check m_{i+j} \leq \check m_i$ if $\check m_i$ is even,
and $\check m_{i+j} \leq \check m_i+1$ if $\check m_i$ is odd.
By Lemma \ref{ERgeq2qf+2}, $\check m_{i+j_i} \geq 2(q_f+1)$ for the last infinitely close singularity
$\check x_{i+j_i}$ introduced by $\check\psi$.
Thus $\check m_i \geq 2q_f+1$ as required.

Now we assume $\check m_i=2q_f+1$.
Note that we have already proved that $\check m_{i+1} \geq 2q_f+1$ in the above.
If $\check m_{i+1}=2q_f+2$,
then $\check x_{i+1}$ must be the unique singular point
of $\check R_{i+1}$ on $\check {\mathcal E}_{i+1} \subseteq \check P_{i+1}$,
and we are done.
Therefore it is enough to derive a contradiction if $\check m_{i+1} = 2q_f+1$.

Let $l$ be the smallest number such that $\check m_{i+l}=2(q_f+1)$,
where we assume that $\check x_{i+j}$ is infinitely close to $\check x_{i+j-1}$ for $j=1,\cdots, l$.
Such an $l$ exists by Lemma \ref{ERgeq2qf+2}.
And $l\geq 2$ if $\check m_{i+1} = 2q_f+1$.
Note that the exceptional curve $\check {\mathcal E}_{i+j}$ is contained in $\check R_{i+j}$
for $1\leq j \leq l$, since $\check m_{i+j-1}$ is odd.
Because $\check m_{i+l}=\check m_{i+l-1}+1$,
$\check x_{i+l}$ must be the unique singular point of $\check R_{i+l}$ on the exceptional curve
$\check {\mathcal E}_{i+l}$.

Let $\check E_{i+l-1}\subseteq \check P_{i+l}$
be the strict transform of $\check {\mathcal E}_{i+l-1} \subseteq \check P_{i+l-1}$,
$D=\check R_{i+l}-(\check E_{i+l-1}+\check {\mathcal E}_{i+l})$,
and $D'$ the image of $D$ in $\check P_{i+l-1}$.
Then $\check x_{i+l}\in \check E_{i+l-1}$, since
$\check x_{i+l}$ is the unique singular point of $\check R_{i+l}$ on the exceptional curve
$\check {\mathcal E}_{i+l}$ and $\check E_{i+l-1}\cap \check {\mathcal E}_{i+l}$
is a singularity of  $\check R_{i+l}$.
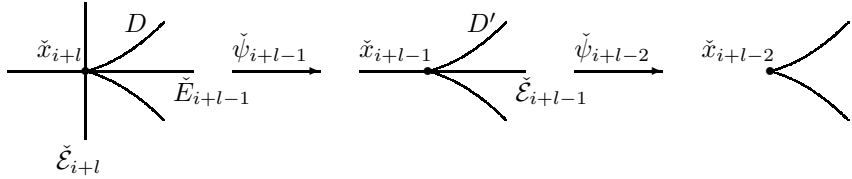
\begin{center}
\setlength{\unitlength}{1.3mm}
\begin{picture}(100,15)
\put(12,3){\line(0,1){14}}
\put(4,10){\line(1,0){19}}
\qbezier(12,10)(16,11)(20,15)
\qbezier(12,10)(16,9)(20,5)
\put(12,10){\circle*{0.8}}
\put(7,11.5){$\check x_{i+l}$}
\put(9,0){$\check {\mathcal E}_{i+l}$}
\put(21,7){$\check {E}_{i+l-1}$}
\put(16,14){$D$}
\put(27,10){\vector(1,0){9}}
\put(27,11.5){$\check \psi_{i+l-1}$}

\put(40,10){\line(1,0){17}}
\qbezier(47,10)(51,11)(55,15)
\qbezier(47,10)(51,9)(55,5)
\put(47,10){\circle*{0.8}}
\put(40,11.5){$\check x_{i+l-1}$}
\put(56,7){$\check {\mathcal E}_{i+l-1}$}
\put(51,14){$D'$}
\put(62,10){\vector(1,0){9}}
\put(62,11.5){$\check \psi_{i+l-2}$}

\qbezier(82,10)(86,11)(90,15)
\qbezier(82,10)(86,9)(90,5)
\put(82,10){\circle*{0.8}}
\put(75,11.5){$\check x_{i+l-2}$}
\end{picture}
\end{center}
Since $\check m_{i+l}=2(q_f+1)$,
$D$ has multiplicity $\check m_{i+l}-2=2q_f$ at $\check x_{i+l}$.
Hence the local intersection
$$I_{\check x_{i+l}}(\check E_{i+l-1},\,D) \geq  2q_f,\qquad
I_{\check x_{i+l}}(\check{\mathcal E}_{i+l},\,D) \geq  2q_f.$$
Note that $D$ is nothing but the strict transform of $D'$,
and $(\check\psi_{i+l-1})^*(\check {\mathcal E}_{i+l-1})=\check E_{i+l-1}+\check{\mathcal E}_{i+l}$. Thus
$$
I_{\check x_{i+l-1}}\big(\check {\mathcal E}_{i+l-1},\,D'\big)=
I_{\check x_{i+l}}\big(\check\psi_{i+l-1}^*(\check {\mathcal E}_{i+l-1}),\,D\big)
=I_{\check x_{i+l}}(\check E_{i+l-1}+\check{\mathcal E}_{i+l},\,D) \geq 4q_f.
$$
Note that $\check \psi_{i+l-2}(D')=\check R_{i+l-2}$,
and the multiplicity $\check m_{i+l-2}$ of $\check R_{i+l-2}$
at $\check x_{i+l-2}$ equals the intersection number $\check {\mathcal E}_{i+l-1}\cdot D'$.
Hence
$$2q_f+1=\check m_{i+l-2}=\check {\mathcal E}_{i+l-1}\cdot D'
\geq I_{\check x_{i+l-1}}(\check {\mathcal E}_{i+l-1},\,D')\geq 4q_f,$$
which is a contradiction, since $q_f\geq 1$.
So we finish the proof.
\end{proof}

By the above lemma, it is easy to get the following:
\begin{corollary}\label{contributeofs_i>2q_f}
The contraction $\check\psi$ is composed of several blowing-ups of singularities of $R$ of type $(2k+1 \to 2k+1)$ with $k\geq q_f$,
or of singularities with multiplicity at least $2(q_f+1)$, i.e., it contributes only to $s_i$ with $i>2q_f$.
\end{corollary}

Now we are in the position to prove Proposition \ref{indeplemma}.
\begin{proof}[Proof of Proposition {\rm \ref{indeplemma}}]~
Let $\check s_{2k+1}$ ($k\geq 1$) be the number of the singularities of $R$
of type $(2k+1 \to 2k+1)$ introduced by $\check\psi$,
and $\check s_{2k}$ ($k\geq 2$) be the number of the singularities of $R$ with multiplicity $2k$ or $2k+1$
introduced by $\check\psi$
neither belonging to the second component of $(2k-1 \to 2k-1)$ type singularities nor to
the first component of $(2k+1 \to 2k+1)$  type singularities.
Then $\check s_i \geq 0$, and by Corollary \ref{contributeofs_i>2q_f} one has
\begin{equation}\label{checksbound}
\check s_{2k+1}=0, \quad \forall~ 1\leq k\leq q_f-1;\qquad\qquad
\check s_{2k}=0,  \quad \forall~ 1\leq k\leq q_f.
\end{equation}
Let $\bar s_i=s_i-\check s_{i}.$
Then for $i>2$, $\bar s_i$ is nothing but the number of the singularities of $R$ introduced by $\bar\psi$ with the corresponding multiplicity or types.
Hence $\bar s_i\geq 0$ for $i>2$;
and one has by \eqref{checksbound} that
\begin{equation}\label{defofbars}
s_{2k+1}=\bar s_{2k+1}, \quad \forall~ 1\leq k\leq q_f-1;\qquad\quad
s_{2k}=\bar s_{2k},  \quad \forall~ 1\leq k\leq q_f.
\end{equation}

According to Lemma \ref{wtR=Gamma_i} and the decomposition of $\psi$ in Figure\,\ref{bimapfactordiagram},
we see that $\ol R$ is contained in the fibers of $\bar h'$,
hence it is semi-negative definite.
By the definition of the $\check s_i$'s and Lemma \ref{numberofcontraction},
there are
$\sum\limits_{k=q_f}^{[g/2]}  \check s_{2k+1}+2\check s_{g+2}$
 isolated $(-2)$-curves contained in $\ol R$.
Thus
$$\ol R^2 \leq -2 \left(\sum_{k=q_f}^{[\frac{g}{2}]}  \check s_{2k+1}+2\check s_{g+2}\right).$$
On the other hand, by definition,
$$\ol R^2=R^2-\sum_{k=q_f}^{[\frac{g}{2}]}4(2k^2+2k+1)\check s_{2k+1}
-\sum_{k=q_f+1}^{[\frac{g+1}{2}]}4k^2 \check s_{2k}
-2(g^2+4g+5)\check s_{g+2}.$$
As $R^2=4L^2=4(g+1)n$ by \eqref{n=delta^2/g+1}, we get
\begin{equation}\label{(g+1)n}
(g+1)n\leq \sum_{k=q_f}^{[\frac{g}{2}]}\left(2k^2+2k+\frac12\right)\check s_{2k+1}
+\sum_{k=q_f+1}^{[\frac{g+1}{2}]}k^2 \check s_{2k}+\frac{(g+1)(g+3)}{2}\check s_{g+2}.
\end{equation}
Hence 
\begin{equation*}
\begin{aligned}
&s_2+\sum_{k=1}^{q_f-1}4k(2k+1) s_{2k+1}+\sum_{k=2}^{q_f}2k(2k-1) s_{2k}\\
\leq\,& \bar s_2+\sum_{k=1}^{[\frac{g}{2}]}4k(2k+1) \bar s_{2k+1}
+\sum_{k=2}^{[\frac{g+1}{2}]}2k(2k-1) \bar s_{2k}+2(g^2+3g+1)\bar s_{g+2}\\
\leq\,& \sum_{k=q_f}^{[\frac{g}{2}]}\frac{(2k+1)(2g+1-2k)}{g+1}\check s_{2k+1}
+\sum_{k=q_f+1}^{[\frac{g+1}{2}]}\frac{2k(g+1-k)}{g+1}\check s_{2k}+(g+1)\check s_{g+2}\\
\leq\,& \sum_{k=q_f}^{[\frac{g}{2}]}\frac{(2k+1)(2g+1-2k)}{g+1} s_{2k+1}
+\sum_{k=q_f+1}^{[\frac{g+1}{2}]}\frac{2k(g+1-k)}{g+1} s_{2k}+(g+1) s_{g+2}.
\end{aligned}
\end{equation*}
The first and last inequalities above follow immediately from the non-negativity of $\bar s_i$'s for $i>2$ and \eqref{defofbars};
and the second one follows from \eqref{n=s_i} and \eqref{(g+1)n}.
The proof is complete.
\end{proof}

\section{Proof of Theorem \ref{maintheorem} and its corollaries}\label{proofofmain}
This section aims to prove our main result
Theorem \ref{maintheorem} and its corollaries.
It is based on \eqref{s2leqDelta} given in Proposition \ref{indeplemma}
and the formulas given in Theorem \ref{xiaotheorem}.

\vspace{0.5mm}
\begin{proof}[Proof of Theorem {\rm\ref{maintheorem}}]~
According to \eqref{boundofq_f}, it is known that $q_f\leq \frac{g+1}{2}$\vspace{0.1cm}.
Recall the definition of $\lambda_{g,q_f}$ in \eqref{defintionoflambda_gq_f}.
If $q_f=0$, then $\lambda_{g,0}=\frac{4(g-1)}{g}$\vspace{0.1cm},
and so \eqref{mainequation} holds by \eqref{slopeinequality}.
Thus we assume $q_f\geq 1$ in the following.

First we prove that
\begin{equation}\label{lowboundequa}
\begin{aligned}
K_f^2~\geq~&\lambda_{g,q_f} \cdot\chi_f
+  \alpha s_{g+2}\\
&+\sum_{k=1}^{q_f-1} \alpha_k s_{2k+1}
+\sum_{k=2}^{q_f}\beta_k s_{2k}
+\sum_{k=q_f}^{[\frac{g}{2}]} \gamma_k s_{2k+1}
+\sum_{k=q_f+1}^{[\frac{g+1}{2}]} \delta_k s_{2k},
\end{aligned}
\end{equation}
where
\begin{equation*}
\left\{
\begin{aligned}
 \alpha=&~\frac{(g-1)(8-\lambda_{g,q_f})}{4},&&\\[0.1cm]
\alpha_k=&~ k^2\lambda_{g,q_f}-(2k-1)^2,&& \forall\, 1\leq k\leq q_f-1,\\[0.1cm]
\beta_k=&~\frac{(k-1)\big(k\lambda_{g,q_f}-4(k-1)\big)}{2},&&\forall\, 2\leq k\leq q_f,\\[0.1cm]
\gamma_k=&~ \frac{8(4k(g-k)-1)-(4k(g-k)+g)\cdot\lambda_{g,q_f}}{4(g+1)},
&&\forall\, q_f\leq k\leq [\frac{g}{2}],\\[0.1cm]
\delta_k=&~ \frac{k(g+1-k)(8-\lambda_{g,q_f})-4(g+1)}{2(g+1)},&&\forall\, q_f+1\leq k\leq [\frac{g+1}{2}].
\end{aligned}\right.
\end{equation*}
Indeed, it is clear that $\lambda_{g,q_f} \geq \frac{4(g-1)}{g}$.
Hence by Theorem \ref{xiaotheorem} and \eqref{s2leqDelta}, one obtains
\begin{eqnarray*}
&&(2g+1)\left(K_f^2-\lambda_{g,q_f}\cdot\chi_f\right)\\
&=& \left(g-1-\frac{g}{4}\cdot\lambda_{g,q_f}\right)s_2
+\left((3g^2-2g-1)-(g^2-2g-1)\cdot\frac{\lambda_{g,q_f}}{4}\right)s_{g+2}\\
&&+\sum_{k=1}^{[\frac{g}{2}]}\big(a_k-k(g-k)\cdot\lambda_{g,q_f}\big)s_{2k+1}
+\sum_{k=2}^{[\frac{g+1}{2}]}\Big(b_k-\frac12k(g+1-k)\cdot\lambda_{g,q_f}\Big)s_{2k}\\
&\geq & \left(g-1-\frac{g}{4}\cdot\lambda_{g,q_f}\right)\cdot\Lambda_h
+\left((3g^2-2g-1)-(g^2-2g-1)\cdot\frac{\lambda_{g,q_f}}{4}\right)s_{g+2}\\
&&+\sum_{k=1}^{[\frac{g}{2}]}\big(a_k-k(g-k)\cdot\lambda_{g,q_f}\big)s_{2k+1}
+\sum_{k=2}^{[\frac{g+1}{2}]}\Big(b_k-\frac12k(g+1-k)\cdot\lambda_{g,q_f}\Big)s_{2k}\\
&=&(2g+1)\left(\alpha s_{g+2}
+\sum_{k=1}^{q_f-1} \alpha_k s_{2k+1}
+\sum_{k=2}^{q_f}\beta_k s_{2k}
+\sum_{k=q_f}^{[\frac{g}{2}]} \gamma_k s_{2k+1}
+\sum_{k=q_f+1}^{[\frac{g+1}{2}]} \delta_k s_{2k}\right),
\end{eqnarray*}
where
$$\begin{aligned}
\Lambda_h\,=\,& \sum_{k=q_f}^{[\frac{g}{2}]}\frac{(2k+1)(2g+1-2k)}{g+1} s_{2k+1}
+\sum_{k=q_f+1}^{[\frac{g+1}{2}]}\frac{2k(g+1-k)}{g+1} s_{2k}+(g+1) s_{g+2}\\
&-\sum_{k=1}^{q_f-1}4k(2k+1) s_{2k+1}-\sum_{k=2}^{q_f}2k(2k-1) s_{2k}.
\end{aligned}$$
Therefore, \eqref{lowboundequa} follows.

To prove \eqref{mainequation},
it suffices to prove that
those coefficients $\alpha$, $\alpha_k$, $\beta_k$, $\delta_k$ and $\gamma_k$ in \eqref{lowboundequa}
are all non-negative by noting that $s_i \geq 0$ for any $i\geq 3$.
It is clear that
\begin{equation}\label{slopelinshi1}
\alpha \geq 0;\qquad \alpha_k>0,~\forall\, 1\leq k\leq q_f-1;
\text{\qquad and~} \beta_k>0,~\forall\, 2\leq k\leq q_f.
\end{equation}

If $q_f \leq \frac{g-1}{2}$, then
\begin{eqnarray*}
\gamma_k&=&\frac{4k(g-k)+g}{(q_f+1)(g-q_f)}-2\\
&\geq& \frac{4q_f(g-q_f)+g}{(q_f+1)(g-q_f)}-2=\frac{2(q_f-1)(g-q_f)+g}{(q_f+1)(g-q_f)}
> 0, \qquad\quad \forall~q_f\leq k \leq [\frac{g}{2}].\\[0.2cm]
\delta_k&=& 2\left(\frac{k(g+1-k)}{(q_f+1)(g-q_f)}-1\right)\\
&\geq& 2\left(\frac{(q_f+1)\big(g+1-(q_f+1)\big)}{(q_f+1)(g-q_f)}-1\right)
=0, \qquad\qquad \forall~q_f+1\leq k \leq [\frac{g+1}{2}].
\end{eqnarray*}
Hence \eqref{mainequation} follows from \eqref{lowboundequa} when $q_f \leq \frac{g-1}{2}$.

If $g$ is even and $q_f=\frac{g}{2}$,
then the last summation in \eqref{lowboundequa} disappears automatically.
So it suffices to show $\gamma_{g/2}\geq 0$.
By definition,
$$\gamma_{g/2}=\frac{1}{4(g+1)}\cdot
\left(8\cdot\Big(4\cdot\frac{g}{2}\cdot\big(g-\frac{g}{2}\big)-1\Big)-
\Big(4\cdot\frac{g}{2}\cdot\big(g-\frac{g}{2}\big)+g\Big)\cdot\frac{8(g-1)}{g}\right)=0.$$
Hence, \eqref{mainequation} also holds in this case.

Finally, if $g$ is odd and $q_f=\frac{g+1}{2}$,
then the last two summations in \eqref{lowboundequa} disappear,
hence we have already showed that $K_f^2 \geq \lambda_{g,q_f}\cdot\chi_f$ by \eqref{slopelinshi1}.
So \eqref{mainequation} holds in this case too.
This completes the proof.
\end{proof}

\vspace{0.1mm}
\begin{proof}[Proof of Corollary {\rm\ref{corollarypfofconjectue}}]~
If $g$ is even and $q_f=\frac{g}{2}$,
or $g$ is odd and $q_f=\frac{g+1}{2}$, then by definition, $$\lambda_{g,q_f}=\frac{4(g-1)}{g-q_f}.$$
If $q_f\leq \frac{g-1}{2}$, i.e., $g-2q_f-1\geq0$, then
$$
\lambda_{g,q_f}-\frac{4(g-1)}{g-q_f}=
\frac{4q_f\big(g-2q_f-1\big)}{(q_f+1)(g-q_f)}\geq 0;
$$
and `$=$' holds only if $q_f=0$ or $\frac{g-1}{2}$.
Therefore, our corollary is a consequence of \eqref{mainequation}.
\end{proof}

\begin{remark}
When $q_f=0$, Corollary \ref{corollarypfofconjectue} is just \eqref{slopeinequality}.
When $q_f=\frac{g+1}{2}$, $\frac{g}{2}$ or $\frac{g-1}{2}$,
Corollary \ref{corollarypfofconjectue} was already obtained by Xiao in \cite{xiao92-0}.
\end{remark}

\vspace{0.1mm}
\begin{proof}[Proof of Corollary {\rm\ref{corollarypfofconjectuexiao}}]~
Let $f:\,S \to B$ be a fibration of genus $g\geq 2$,
which is not locally trivial and $\lambda_f<4$.
We need to prove that $f_*\omega_{S/B}$ has no locally free quotient of degree zero.
By \cite[Theorem\,1]{bz00},
we may assume that $f$ is a hyperelliptic fibration.
Recall that by \eqref{fujitadecom} we have the following decomposition:
\begin{equation}\label{fujitadecom'}
f_*\omega_{S/B}=\mathcal A \oplus \mathcal F_1 \oplus \cdots
\oplus \mathcal F_k \oplus \mathcal O_B^{\oplus q_f},
\end{equation}
with $\mathcal A$ ample, $\mathcal F_i$ irreducible unitary
and $\dim H^1(B, \Omega^1_B(\mathcal F_i))=0$ for $1\leq i \leq k$.
Since $\lambda_f<4$, we have that $q_f=0$ by Corollary \ref{corollarypfofconjectue}.
Clearly $\mathcal A$ has no non-trivial locally free quotient of degree zero.
Hence it suffices to prove that $\mathcal F_i=0$ in the above decomposition \eqref{fujitadecom'}.

Assume that $\mathcal F_i \neq 0$ for some $i$.
By construction, $\mathcal F_i$ corresponds to a unitary representation of the fundamental group
$$\rho_i:~\pi_1(B) \lra U(r_i),\qquad\qquad\text{where~} r_i=\rank \mathcal F_i.$$
If the image of $\rho_i$ is finite, then, after a suitable finite \'etale base change,
$\mathcal F_i$ becomes trivial, which implies that $q_f>0$ after such a finite \'etale base change.
However, it is a contradiction by Corollary \ref{corollarypfofconjectue}, since
the slope does not change under any finite \'etale base change.
Hence we may assume that $\rho_i$ has infinite image.

By the stable reduction theorem (cf. \cite{artinwinters,delignemumford}),
there exists a base change $\phi:\wt B \to B$ of finite degree, possibly ramified,
such that the pull-back fibration $\tilde f: \wt S \to \wt B$ is semi-stable.
According to \cite[Theorem\,4.7]{luzuo14} (see also Theorem \ref{thmFtrivial} in the appendix),
applying possibly a further base change, we may assume that the Fujita decomposition of $\tilde f_*\omega_{\wt S/ \wt B}$ is as follows,
\begin{equation}\label{pfofconj1}
\tilde f_*\omega_{\wt S/ \wt B} = \wt {\mathcal A} \oplus \mathcal O_{\wt B}^{\oplus q_{\tilde f}},
\quad\text{with $\wt {\mathcal A}$ ample}.
\end{equation}
Here the pull-back fibration $\tilde f:\, \wt S \to \wt B$ is constructed as follows.
Let $S_1$ be the resolution of singularities of $S\times_{B}\wt B$.
Then $\tilde f:\, \wt S \to \wt B$ is just the relatively minimal model of $S_1$.
\begin{center}\mbox{}
\xymatrix{
\wt S \ar@{<-}[rr]^-{\theta} \ar[d]_-{\tilde f} && S_1 \ar[rr]^-{\Phi_1} \ar[d]_-{f_1}
&& S\times_{B}\wt B \ar[rr]^-{\Phi_0} \ar[d] && S \ar[d]^-{f}\\
\wt B \ar@{=}[rr]  && \wt B \ar@{=}[rr] && \wt B \ar[rr]^-{\phi} && B}
\end{center}

It is known that (cf. \cite[p.231]{tan94}) there is an inclusion
$\tilde f_*\omega_{\wt S/ \wt B} \subseteq \phi^* f_*\omega_{S/B}$.
As $$\rank \tilde f_*\omega_{\wt S/ \wt B} =\rank \phi^* f_*\omega_{S/B}=g,$$
the quotient $Q:=(\phi^* f_*\omega_{S/B})\big/(\tilde f_*\omega_{\wt S/ \wt B})$ is a torsion sheaf.
By projection, we get a morphism:
$${\rm pr}_i:~\tilde f_*\omega_{\wt S/ \wt B} \lra \phi^*\mathcal F_i.$$
Hence the quotient $Q_i:=(\phi^*\mathcal F_i)\big/{\rm pr}_i(\tilde f_*\omega_{\wt S/ \wt B})$ is also a torsion sheaf.
Note that $\deg \left({\rm pr}_i(\tilde f_*\omega_{\wt S/ \wt B})\right) \geq 0$,
since it is a quotient of $\tilde f_*\omega_{\wt S/ \wt B}$.
Note also that $\deg Q_i\geq 0$, and
$$0=\deg \phi^*\mathcal F_i =\deg Q_i+\deg \left({\rm pr}_i(\tilde f_*\omega_{\wt S/ \wt B})\right).$$
Hence we obtain that $Q_i$ is zero, and ${\rm pr}_i$ is surjective.

By construction, $\phi^*\mathcal F_i$ comes from the following unitary representation
$$\xymatrix{
  \pi_1(\wt B) \ar[rr]^-{\tilde\rho_i} \ar[dr]_-{\phi_*}
                &  &    U(r_i)  \\
                & \pi_1(B)  \ar[ru]_-{\rho_i}                 }$$
It follows that $\tilde\rho_i$ has infinite image,
since $\rho_i$ has infinite image and $\phi_*\left(\pi_1(\wt B)\right)$ has finite index in $\pi_1(B)$.

Since $\wt {\mathcal A}$ in \eqref{pfofconj1} is ample, it maps to zero by ${\rm pr}_i$.
Therefore we have a surjective morphism:
$${\rm pr}_i:~\mathcal O_{\wt B}^{\oplus q_{\tilde f}} \lra \phi^*\mathcal F_i.$$
Note that the trivial bundle $\mathcal O_{\wt B}^{\oplus q_{\tilde f}}$ corresponds to the trivial representation (hence also unitary representation) of $\pi_1(\wt B)$.
Hence by \cite{ns65}, $\phi^*\mathcal F_i$ is a direct summand of $\mathcal O_{\wt B}^{\oplus q_{\tilde f}}$,
which implies that the representation $\tilde\rho_i$ corresponding to $\phi^*\mathcal F_i$ is also trivial.
%
This gives a contradiction,
since the representation $\tilde\rho_i$ has infinite image by construction.
This completes the proof.
\end{proof}

\section{Examples}\label{examples}
In this section, we construct examples to show that the bound \eqref{mainequation} is sharp.
\begin{example}\label{ex(g+1)=m(q_f+1)}
Hyperelliptic fibration $f$ of genus $g$ with relative irregularity $q_f$
satisfying $g+1=m(q_f+1)$ for some $m\geq 2$, and
$$\lambda_f=\frac{K_f^2}{\chi_f}=\lambda_{g,q_f},
\qquad \text{where $\lambda_{g,q_f}$ is defined in \eqref{defintionoflambda_gq_f}}.$$

Let $P=\bbp_{\bbp^1}\big(\mathcal O_{\bbp^1}\oplus \mathcal O_{\bbp^1}(e)\big)$
be the rational ruled surface
with invariant $e\geq 1$.
Let $$h:\,P\lra B \triangleq\bbp^1$$ be the ruling, $\Gamma \subseteq P$ a general fiber of $ h$,
and $C_0 \subseteq P$ the unique section with self-intersection $C_0^2=-e$.
According to \cite[\S\,V-2]{hartshorne}, the divisor $mC_0+b\Gamma$ is very ample
if and only if $b> me$.
Let $mC_0+b_0\Gamma$ be a very ample divisor.
Then by Bertini's theorem (cf. \cite[\S\,II-8]{hartshorne}),
a general member $D\in |mC_0+b_0\Gamma|$ is smooth and
any two general members $D_1,\,D_2\in |mC_0+b_0\Gamma|$ intersect with each other transversely.
Let $D,\,D'$ be two general members in $|mC_0+b_0\Gamma|$,
and $\Lambda$ the pencil generated by $D$ and $D'$.
Then $\Lambda$ defines a rational map $\varphi_{\Lambda}:\,P \dashrightarrow \mathbb P^1$.
By blowing up the base points of $\Lambda$, we get a fibration $\tilde h':\,\wt P \to \bbp^1$,
$$\xymatrix{
  \wt P \ar[rr]^-{\psi} \ar[dr]_-{\tilde h'}
                &  &    P  \ar@{-->}[dl]^-{\varphi_{\Lambda}}   \\
                & \bbp^1                 }$$
where $\psi:\, \wt P \to P$ is composed of blowing-ups centered at the base points of $\Lambda$.
Let $\wt \Gamma'$ be a general fiber of $\tilde h'$, $K_{\wt P}$ the canonical divisor of $\wt P$.
Then
\begin{equation}\label{examplinshi1}
K_{\wt P}^2=8-x,\qquad K_{\wt P}\cdot \wt\Gamma' =\frac{(m-1)x}{m}-2m,
\qquad \left(\wt \Gamma'\right)^2=0,
\end{equation}
where
$x=(mC_0+b_0\Gamma)^2$
is the number of blowing-ups contained in $\psi$.
Let $\Delta \subseteq \bbp^1$ be a set of $2(q_f+1)$ general points,
and
$\wt R=(\tilde h')^*(\Delta)$
the corresponding fibers of $\tilde h'$.
Let $\pi':\,B'\to \bbp^1$ be the double cover ramified over $\Delta$,
and $\wt S$ be the normalization of the fiber-product $\wt P\times_{\bbp^1} B'$:
\begin{center}\mbox{}
\xymatrix{
\wt S \ar@{->}[rr]^-{\tilde\pi} \ar@{->}[d]_-{\tilde f'} && \wt P \ar@{->}[d]^-{\tilde h'}\\
B' \ar@{->}[rr]^-{\pi'} && \bbp^1
}
\end{center}
Note that if $\Delta$ is general on $\bbp^1$, then $\wt R$ is both reduced and smooth.
Hence $\wt S$ is also smooth.
Note that $\wt R \equiv (2q_f+2)\wt \Gamma'$,
where $\wt \Gamma'$ is a general fiber of $\tilde h'$.
So by the standard formulas for double covers (cf. \cite[\S\,V.22]{bhpv}) and \eqref{examplinshi1}, one gets
\begin{eqnarray*}
K_{\wt S}^2&=& 2\left(K_{\wt P}+\frac12\wt R\right)^2
=\left(\frac{4(m-1)(q_f+1)}{m}-2\right)x-\big(8m(q_f+1)-16\big),\\
\chi(\mathcal O_{\wt S})&=&2\chi(\mathcal O_{\wt P})
+\frac12\left(K_{\wt P}+\frac12\wt R\right)\cdot\frac{\wt R}{2}
=\frac{(m-1)(q_f+1)}{2m}x-\big(m(q_f+1)-2\big).
\end{eqnarray*}
The ruling $ h:\, P\to B \cong \bbp^1$ induces a fibration $\tilde h:\, \wt P \to B$ and hence
a fibration $f:\,\wt S \to B$.
$$\xymatrix{
  \wt S \ar[rr]^-{\tilde\pi} \ar[drr]_-{f}
                &  &    \wt P  \ar[d]^-{\tilde h}\ar[rr]^-{\psi} &&  P  \ar[dll]^-{ h} \\
                && B\cong \bbp^1&&                 }$$
It is easy to show that the induced map $\tilde h|_{\wt R}:\,\wt R \to B$
is of degree $m\cdot 2(q_f+1)=2(g+1)$,
hence $f$ is a hyperelliptic fibration of genus $g$.
By construction, the relative irregularity of $f$ is just $q_f=g(B')$, and
\begin{eqnarray*}
K_f^2  &=& K_{\wt S}^2-8(g-1)(g(B)-1)=\left(\frac{4(m-1)(q_f+1)}{m}-2\right)x,\\
\chi_f &=&\chi(\mathcal O_{\wt S})-(g-1)(g(B)-1)=\frac{(m-1)(q_f+1)}{2m}x.
\end{eqnarray*}
Actually, $f$ is relatively minimal.
To see this, let $R$ be the image of $\wt R$ in $P$.
Then the singular points of $R$ are all of multiplicity $2(q_f+1)$.
Hence $s_{2k+1}=0$ for all $k\geq 1$, which implies that $\wt S$
is relatively minimal by Lemma \ref{numberofcontraction}.
Hence $f$ is a relatively minimal locally non-trivial hyperelliptic fibration of genus $g$
with relative irregularity $q_f=\frac{g+1}{m}-1\leq \frac{g-1}{2}$, and
$$\lambda_f=\frac{K_f^2}{\chi_f}=8-\frac{4m}{(m-1)(q_f+1)}
=8-\frac{4(g+1)}{(g-q_f)(q_f+1)}=\lambda_{g,q_f}.$$
\end{example}

\begin{remarks}\label{remarkg=3q_f=1}
(i). Let $g=3$ and $m=2$ in the above example. Then we get a hyperelliptic fibration $f$ of genus $3$ with
relative irregularity $q_f=1$ and slope $\lambda_f=4$.

(ii). According to Bertini's theorem (cf. \cite[\S\,II-8]{hartshorne}),
for a general member $D\in |mC_0+b_0\Gamma|$,
the projection of $ h|_D:\,D\to B$ has at most simply ramified points.
Hence if $\Delta$ is general enough, then the fibration $f$
obtained in the above example is semi-stable.
\end{remarks}

\begin{example}\label{exq_f=g/2or(g+1)/2}
Hyperelliptic fibration $f$ of genus $g$ with
$$\left\{
\begin{aligned}
&q_f=\frac{g}{2},&\quad&\lambda_f=\frac{K_f^2}{\chi_f}=\frac{8(g-1)}{g}, &\qquad& \text{if $g$ is even};\\
&q_f=\frac{g+1}{2},&&\lambda_f=\frac{K_f^2}{\chi_f}=8, && \text{if $g$ is odd}.
\end{aligned}
\right.$$

Let $F$ be the hyperelliptic curve of genus $g$ defined by $u^2=v^{2g+2}-1$,
and $ \tau_1$ be an involution of $F$ defined by $ \tau_1(u,v)=(-u,-v)$.
Then $ \tau_1$ has exactly two fixed points if $g$ is even,
and $ \tau_1$ has no fixed point if $g$ is odd.
Let $\phi:\,\wt B \to B$ be a double cover between two projective curves of genus
$\tilde b=g(\wt B)$ and $b=g(B)$ respectively,
$\Sigma\subseteq B$ the branch divisor,
$\wt\Sigma=\phi^{-1}(\Sigma)$,
and $ \tau_2$ the induced involution of $\wt B$ such that $B=\wt B/\langle \tau_2\rangle$.
Let $\tau=( \tau_1, \tau_2)$ be an involution of $X=F\times \wt B$ defined by
$\tau(p,q)=\big( \tau_1(p), \tau_2(q)\big)$,
where $p\in F$ and $q\in \wt B$.
Then $X/\langle \tau\rangle$ has a natural fibration of genus $g$ over $B$.
Let $f:\,S\to B$ be the relatively minimal smooth model of $X/\langle \tau\rangle$ as follows.
\begin{center}\mbox{}
\xymatrix{
X \ar@{->}[rr]^-{\Phi} \ar@{->}[d]_-{\tilde f} && X/\langle \tau\rangle \ar@{->}[d]^-{f'}
\ar@{<-->}[rr]&& S\ar@{->}[dll]^-{f}\\
\wt B \ar@{->}[rr]^-{\phi} && B&&
}
\end{center}

Assume $\Sigma \neq \emptyset$. Then we see that $f$ is a non-trivial hyperelliptic fibration.
If $g$ is even, then $\tau$ has exactly two fixed points over each fiber in $\tilde f^*(\wt\Sigma)$,
and so $X/\langle \tau\rangle$ has only rational singularities of type $A_1$ (cf. \cite[\S\,III-3]{bhpv});
if $g$ is odd, then $\tau$ is fixed-point-free,
and hence $S=X/\langle \tau\rangle$ is already smooth and relatively minimal.
Let $|\Sigma|$ be the number of points in $\Sigma$.
Then one can compute that
\begin{eqnarray*}
K_f^2&=&2(g-1)\cdot|\Sigma|,\\
\chi_f&=&
\left\{
\begin{aligned}
&\frac{\,g\,}{4}\cdot |\Sigma|,   &\quad& \text{if $g$ is even};\\[0.1cm]
&\frac{g-1}{4}\cdot |\Sigma|, &\quad& \text{if $g$ is odd}.
\end{aligned}
\right.\end{eqnarray*}
Therefore, we obtain hyperelliptic fibrations with the required slopes.
To compute the relative irregularity, we consider another projection of $X$,
i.e., $\tilde h:\,X \to F$.
It induces a fibration $h':\,X/\langle \tau\rangle \to B'=F/\langle \tau_1\rangle$,
and hence also a fibration
$h:\,S \to B'$ with
$$g(B')=\frac{g}{2}, ~~\text{if $g$ is even,} \qquad \text{and} \qquad
g(B')=\frac{g+1}{2}, ~~\text{if $g$ is odd.}$$
In particular, $q_f \geq g(B')$.
Combining this with \eqref{boundofq_f},
we see that $q_f=g(B')$ as required.
\end{example}

\begin{remark}\label{remarkg=2q_f=1}
Taking $g=2$ in the above example, we get a hyperelliptic fibration of genus $2$ with
relative irregularity $q_f=1$ and slope $\lambda_f=4$.
\end{remark}

\section*{Appendix}
\setcounter{section}{1}
\setcounter{theorem}{0}
\setcounter{equation}{0}
\renewcommand{\theequation}{\Alph{section}-\arabic{equation}}
\renewcommand{\thetheorem}{\Alph{section}.\arabic{theorem}}
In this appendix, we would like to prove the following theorem on the unitary part of a semi-stable hyperelliptic fibration:

\begin{theorem}\label{thmFtrivial}
Let $f: S\to B$ be a semi-stable hyperelliptic fibration of curves of genus $g\geq 2$ with relative irregularity $q_f$.
Consider the following Fujita decomposition {\rm(}cf. \cite{fujita78b,kollar87,cd14}{\rm):}
$$f_*\omega_{S/B}=\mathcal A^{1,0} \oplus \mathcal F^{1,0},\quad\text{where $\mathcal A^{1,0}$ is ample and $\mathcal F^{1,0}$ is unitary.}$$
Then after passing to a suitable finite \'etale base change, one has
\begin{equation}\label{eqnFtrivial}
\rank \mathcal F^{1,0}=q_{f}.
\end{equation}
\end{theorem}
Note that the above theorem has been proven in \cite[Theorem\,4.7]{luzuo14}.
The proof there used a general lemma (cf. \cite[Lemma\,7.1]{luzuo14}) regarding the global invariant cycle with unitary locally constant coefficient, which may be of independent interest.
However, the theorem is purely on fibred surfaces.
Hence we would like to give another proof without using such a lemma.
Instead, we prove a weak result about the lifting property of the unitary part $\mathcal F^{1,0}$ (cf. Lemma \ref{lemmalifting}), which is enough for our purpose.
The rest of the arguments is parallel to that of \cite[Theorem\,4.7]{luzuo14}.

\addtocounter{theorem}{-1}
\renewcommand{\thetheorem}{\Alph{section}.\arabic{theorem}$'$}
{\bf First reduction.}
We first prove that Theorem \ref{thmFtrivial} can be reduced to the following theorem.
\begin{theorem}\label{thmFtrivial'}
Let $f$ be the same as in Theorem {\rm\ref{thmFtrivial}}.
Then after passing to a suitable finite (not necessarily \'etale) base change, one has
\begin{equation}\label{eqnFtrivial'}
\rank \mathcal F^{1,0}=q_{f}.
\end{equation}
\end{theorem}
\begin{proof}[Proof of Theorem {\rm\ref{thmFtrivial}} based on Theorem {\rm\ref{thmFtrivial'}}]
Note that the bundle $\mathcal F^{1,0}$ comes from a unitary representation of the fundamental group
$$\rho:~\pi_1(B) \lra U(r), \qquad \text{~with~}r=\rank \mathcal F^{1,0}.$$
It follows that $\rank \mathcal F^{1,0}=q_{f}$ if and only if the representation $\rho$ is trivial.
Note that for any finite subset $\Sigma\subseteq B$, the restriction $\mathcal F^{1,0}\big|_{B\setminus\Sigma}$ corresponds to the unitary representation $\rho_{\Sigma}$:
$$
\xymatrix{
\pi_1\big(B\setminus\Sigma\big) \ar[rr]^-{\rho_\Sigma}\ar[dr]_-{i_*}&& U(r)\\
&\pi_1(B)\ar[ur]_-{\rho}&
}$$
By Theorem \ref{thmFtrivial'}, there exists a subset $\Sigma$ such that $\rho_\Sigma$ has finite image.
Because $\rho_\Sigma$ factors through $\pi_1(B)$ and $i_*$ is surjective, one gets that $\rho$ has also finite image.
It implies that $\mathcal F^{1,0}$ becomes trivial after a suitable finite \'etale base change, i.e.,
the equality \eqref{eqnFtrivial} holds after a suitable finite \'etale base change.
\end{proof}

By the above reduction, it suffices to prove Theorem \ref{thmFtrivial'}.
We first recall some general knowledge about the variation of Hodge structure for a fibration $f$ and prove a lifting property about $\mathcal F^{1,0}$ in Lemma \ref{lemmalifting}.

\vspace{0.2cm}
Let $\Upsilon\subset S$ be the singular fibers over the degeneration locus $\Delta\subset B$.
Consider the weight one variation of Hodge structures given by $R^1f_*\mathbb Z_{S\setminus \Upsilon}$.
Taking the graded sheaves, one obtains the (logarithmic) Higgs bundle (cf. \cite[\S\,4]{simpson})
$$(E,~\theta)=(E^{1,0}\oplus E^{0,1},~\theta)$$
with $E^{1,0}=f_*\Omega^1_{S/B}(\log \Upsilon)\cong f_*\omega_{S/B}$ and $E^{0,1}=R^1f_*\mathcal O_S$. The Higgs field $\theta$ is given by the edge morphism
$$f_*\Omega^1_{S/B}(\log \Upsilon) \lra R^1f_*\mathcal O_S \otimes \Omega_B^1(\log \Delta)$$
of the tautological sequence
\begin{equation}\label{tautologicalseq}
0\lra f^*\Omega^1_B(\log \Delta) \lra \Omega^1_S(\log \Upsilon)\lra \Omega^1_{S/B}(\log \Upsilon)\lra 0.
\end{equation}
By \cite{fujita78b,kollar87,cd14}, $E$ can be decomposed as a direct sum $\mathcal A\oplus \mathcal F$ of Higgs bundles with
$E^{1,0}\cap \mathcal A$ is ample and $\mathcal F$ flat, hence for $\mathcal A^{i,j}=E^{i,j}\cap \mathcal A$ and $\mathcal F^{i,j}=E^{i,j}\cap \mathcal F$, the Higgs bundle $E$ decomposes as
$$(E^{1,0}\oplus E^{0,1},~\theta)=\left(\mathcal A^{1,0}\oplus \mathcal A^{0,1},~\theta|_{\mathcal A^{1,0}}\right) \bigoplus \left(\mathcal F^{1,0}\oplus \mathcal F^{0,1},~0\right).$$
In particular, $\mathcal F^{1,0} \subseteq M :={\rm Ker\,}(\theta) \subseteq f_*\Omega^1_{S/B}(\log \Upsilon)$,
and we have the following exact sequence induced by \eqref{tautologicalseq}:
\begin{equation}\label{exactseq}
0\lra\Omega^1_B(\log \Delta) \lra f_*\Omega^1_S(\log \Upsilon)\overset{\varsigma}\lra M\lra 0.
\end{equation}
By \cite{fujita78}, one has a further decomposition of $\mathcal F^{1,0}$:
\begin{equation}\label{fujitadecomappendix}
f_*\Omega^1_{S/B}(\log \Upsilon) \cong f_*\omega_{S/B}=\mathcal A^{1,0}\oplus \mathcal F^{1,0}=\mathcal A^{1,0}\oplus \mathcal F_u \oplus \mathcal O_B^{\oplus q_f},
\end{equation}
with $\mathcal A^{1,0}$ ample, $\mathcal F_u$ unitary and $h^0(B, \mathcal F_u^{\vee})=h^1(B, \Omega^1_B(\mathcal F_u))=0$.
Note that $h^0(B, \mathcal F_u^{\vee})=0$ if and only if $h^0(B, \mathcal F_u)=0$, since $\mathcal F_u$ is unitary.

\renewcommand{\thetheorem}{\Alph{section}.\arabic{theorem}}
\begin{lemma}\label{lemmalifting}
Let $i:\,\mathcal F_u \hookrightarrow M$ be the induced inclusion. Then there exists an injective sheaf map $\iota:\,\mathcal F_u \hookrightarrow f_*\Omega^1_S(\log \Upsilon)$
such that $i=\varsigma\circ\iota$, i.e., we have the following commutative diagram:
$$
\xymatrix@M=0.15cm{
&&& \mathcal F_u \ar@{^(->}[d]^-{i}\ar@{_(->}[dl]_-{\exists\, \iota}&\\
0\ar[r] & \Omega^1_B(\log \Delta)\ar[r] & f_*\Omega^1_S(\log \Upsilon) \ar[r]^-{\varsigma} & M \ar[r] &0}$$
\end{lemma}
\begin{proof}
The exact sequence \eqref{exactseq} gives an element $\alpha$ in
${\rm Ext}^1(M, \Omega^1_B(\log \Delta))$.
The existence of $\iota$ such that $i=\varsigma\circ\iota$ is equivalent to
the vanishing of $\alpha$ under the natural map
$${\rm Ext}^1(M, \Omega^1_B(\log \Delta))
\overset{i^{\vee}} \lra {\rm Ext}^1(\mathcal F_u, \Omega^1_B(\log \Delta)).$$
Note that
\begin{eqnarray*}
{\rm Ext}^1(\mathcal F_u, \Omega^1_B(\log \Delta))
&\cong&{\rm Ext}^1(\mathcal O_B, \mathcal F_u^{\vee}\otimes\Omega^1_B(\log \Delta))\\
&\cong&H^1(B, \mathcal F_u^{\vee}\otimes\Omega^1_B(\log \Delta))\\
&\cong&H^0(B, \mathcal F_u(-\Delta))^{\vee}=0.
\end{eqnarray*}
The last equality is due to the fact that $h^0(B, \mathcal F_u)=0$ and $\Delta$ is effective.
Therefore, we obtain a map $\iota$ such that $i=\varsigma\circ\iota$. The injectivity of $\iota$ follows from that of $i$.
\end{proof}
\begin{remark}
The above lemma is weaker than \cite[Corollary\,7.2]{luzuo14},
in which it is proven based on \cite[Lemma\,7.1]{luzuo14} that there exists a lifting $\tilde\iota:\,\mathcal F^{1,0} \hookrightarrow f_*\Omega^1_S(\log \Upsilon)$
for the inclusion $\mathcal F^{1,0} \hookrightarrow M$ and the image of $\tilde\iota$ is actually contained in $f_*\Omega^1_S\subseteq f_*\Omega^1_S(\log \Upsilon)$.
\end{remark}

Now we restrict ourselves to the situation that $f:S\to B$ is a semi-stable hyperelliptic fibration.
Let $\sigma$ be the hyperelliptic involution of $f$, $\vartheta:\, \wt S \to S$ the composition of all the blowing-ups of isolated fixed points of $\sigma$,
and $\tilde\pi:\,\wt S \to \wt P$ the induced double cover with the following diagram:
$$\xymatrix{
&\wt S \ar[dl]_-{\vartheta}\ar[dr]^-{\tilde \pi}\ar[dd]^-{\tilde f}& \\
S \ar[dr]_-{f}&& \wt P \ar[dl]^-{\tilde h}\\
&B&
}
$$

\begin{lemma}\label{lemmacontained}
Let $\wt \Upsilon$ be the strict inverse image of $\Upsilon$ in $\wt S$. Then
$$\vartheta^*\left(\Omega^1_S(\log \Upsilon)\right) \subseteq \Omega^1_{\wt S}(\log \wt\Upsilon).$$
\end{lemma}
\begin{proof}
It is enough to consider the case when $\vartheta:\,\wt S \to S$ consists of exactly one blowing-up.
Let $E$ be the exceptional curve.
Then \[\vartheta^*\left(\Omega^1_S(\log \Upsilon)\right)=\Omega^1_{\wt S}\big(\log (\wt\Upsilon+E)\big)(-E)
\subseteq \Omega^1_{\wt S}(\log \wt\Upsilon). \qedhere\]
\end{proof}

The lifting $\iota$ obtained in Lemma \ref{lemmalifting} gives rise to a morphism
$$f^*{\mathcal F_u} \hookrightarrow f^*f_*\Omega^1_S(\log \Upsilon) \lra \Omega^1_S(\log \Upsilon).$$
Hence by Lemma \ref{lemmacontained}, one obtains a morphism
$$\tilde f^*{\mathcal F_u} =\vartheta^*f^*{\mathcal F_u}  \lra \Omega^1_{\wt S}(\log \wt \Upsilon). $$
Equivalently, we obtain an element
$\eta\in H^0\big(\wt S,\, \Omega^1_{\wt S}(\log \wt\Upsilon)\otimes \tilde f^*\mathcal F_u^{\vee}\big),$
hence also an element
$$\tilde\pi_*\eta\in H^0\Big(\wt P, \tilde\pi_*\big(\Omega^1_{\wt S}(\log \wt\Upsilon)\otimes \tilde f^*\mathcal F_u^{\vee}\big)\Big)
\cong H^0\Big(\wt P, \tilde\pi_*\big(\Omega^1_{\wt S}(\log \wt\Upsilon)\big)\otimes  \tilde h^*\mathcal F_u^{\vee}\Big).$$
So one gets a sheaf map
$$\tilde h^*\mathcal F_u \lra \tilde\pi_*\left(\Omega^1_{\wt S}(\log \wt\Upsilon)\right).$$
The Galois group ${\rm Gal}(\wt S/\wt P)\cong \mathbb Z_2$ acts on $\tilde\pi_*\left(\Omega^1_{\wt S}(\log \wt\Upsilon)\right).$
One therefore obtains the eigenspace decomposition
$$\tilde h^*(\mathcal F_u)\lra \tilde\pi_*\left(\Omega^1_{\wt S}(\log \wt\Upsilon)\right)_1,
\qquad  \tilde h^*(\mathcal F_u)\lra \tilde\pi_*\left(\Omega^1_{\wt S}(\log \wt\Upsilon)\right)_{-1}.$$

\begin{lemma}\label{LR=0}
Assume that the fixed locus of the hyperelliptic involution $\sigma$ consists of $2g+2$ disjoint sections and possibly some isolated points.
Then the image of the map
$$\varrho:~ \tilde h^*(\mathcal F_u)\lra \tilde\pi_*\left(\Omega^1_{\wt S}(\log \wt\Upsilon)\right)_{-1}$$
is an invertible subsheaf $\mathcal M$ such that $\mathcal M$ is nef and $\mathcal M^2=0$.
Let $D$ be any component of the branch divisor $\wt R \subseteq \wt P$ of the double cover $\tilde \pi: \wt S\to \wt P$. Then $\mathcal M\cdot D=0$.
\end{lemma}
\begin{proof}
First of all, we show that $\varrho\neq 0$.
By assumption, the strict inverse image $\tilde\pi^{-1}\big(\wt R\big)$,
which is the fixed locus of the hyperelliptic involution $\tilde\sigma$ on $\wt S$, consists of $2g+2$ disjoint sections and the exceptional curves of $\vartheta$.
Hence it is easy to see that $\wt \Upsilon$ intersects $\pi^{-1}\big(\wt R\big)$ transversely.
Let $\Upsilon'\subseteq \wt P$ be the image of $\wt \Upsilon$, and $\wt L\in \Pic(\wt P)$ such that $\wt L^{\otimes2}\cong \mathcal O_{\wt P}(\wt R)$ defines the double cover $\tilde\pi$.
Then
\begin{equation}\label{pi_*fenjie}
\left\{\begin{aligned}
\tilde\pi_*\left(\Omega^1_{\wt S}(\log \wt\Upsilon)\right)_1&~=~\Omega_{\wt P}^1(\log \Upsilon'),\\
\tilde\pi_*\left(\Omega^1_{\wt S}(\log \wt\Upsilon)\right)_{-1}&~=~\Omega_{\wt P}^1\left(\log (\Upsilon'+\wt R)\right)\otimes \wt L^{-1},
\end{aligned}\right.
\end{equation}
Note that $\tilde h_*\left(\Omega_{\wt P}^1(\log \Upsilon')\right)=0$, since $h^0\left(\wt \Gamma,\Omega_{\wt P}^1(\log \Upsilon')\big|_{\wt \Gamma}\right)=0$ for any general fiber $\wt \Gamma$ of $\tilde h$.
Note also that the induced map
$$\begin{aligned}
\mathcal F_u=\tilde  h_* \tilde h^*\mathcal F_u =\tilde f_*\tilde f^* \mathcal F_u\lra &~\tilde f_*\Omega^1_{\wt S}(\log \wt\Upsilon)=\tilde h_*\left(\tilde\pi_*\big(\Omega^1_{\wt S}(\log \wt\Upsilon)\big)\right)\\
&\hspace{-0.2cm}=\tilde h_*\left(\Omega_{\wt P}^1(\log \Upsilon')\right) \oplus \tilde h_*\bigg(\Omega_{\wt P}^1\left(\log (\Upsilon'+\wt R)\right)\otimes \wt L^{-1}\bigg)
\end{aligned}$$
is the composition of the inclusions $\mathcal F_u \overset{\iota}\hookrightarrow f_*\Omega^1_{S}(\log \Upsilon)
\hookrightarrow \tilde f_*\Omega^1_{\wt S}(\log \wt\Upsilon)$, where the later inclusion follows from Lemma \ref{lemmacontained}.
On the other hand, the second component of the above map is actually given by $\tilde h_*\varrho$.
Hence in particular, $\varrho\neq0$.

Second, we prove that the image of $\varrho$ is a subsheaf of rank one. Otherwise, it is of rank two, and so the second wedge product
$$\wedge^2 \tilde  h^* \mathcal F_u\overset{\wedge^2\varrho}\lra \wedge^2\left(\tilde\pi_*\left(\Omega^1_{\wt S}(\log \wt\Upsilon)\right)_{-1}\right)=\omega_{\wt P}(\Upsilon')$$
is a non-zero map. On the other hand, the image $\mathcal C$ of this map
is a quotient sheaf of  $\wedge^2\tilde h^*\mathcal F_u$ coming from a unitary representation.
Note that for any morphism $\alpha: C \to \wt P$ from a smooth complete curve $C$, $\alpha^*\big(\wedge^2\tilde h^*\mathcal F_u\big)$
is poly-stable of slope zero since it comes from a unitary representation (cf. \cite{ns65}), which implies that $\wedge^2\tilde h^*\mathcal F_u$ is semi-positive.
Therefore, as a quotient of $\wedge^2\tilde h^*\mathcal F_u$, $\mathcal C$ is also semi-positive.
Here we recall that a locally free sheaf $\mathcal E$ on $\wt P$ is called semi-positive, if for any morphism $\alpha: C \to \wt P$ from a smooth complete curve $C$,
the pulling-back $\alpha^*\mathcal E$ has no quotient line bundle of negative degree.
Since $\omega_{\wt P}(\Upsilon')\cdot \wt \Gamma=-2$ for any general fiber $\wt \Gamma$ of $\tilde h$, $\omega_{\wt P}(\Upsilon')$ can not contain any non-zero semi-positive subsheaf.
So the image of $\varrho$ is a rank one subsheaf $\mathcal M\otimes I_{Z}$,
where $\mathcal M$ is an invertible subsheaf and $\dim Z=0$.

We claim that $Z=\emptyset$, and hence the image of $\varrho$ is an invertible subsheaf.
Indeed, if $Z\neq \emptyset$, by a suitable blowing-up $\psi:\,X\to \wt P$, we may assume that the image $\psi^*\varrho\left(\psi^* \tilde h^*(\mathcal F_u)\right)$ is $\psi^*(\mathcal M)\otimes (-E)$,
where $E(\neq 0)$ is a suitable combination of the exceptional curves of $\psi$.
As $\mathcal F_u$ comes from a unitary representation, we get $\psi^*(\mathcal M)\otimes (-E)$ is semi-positive and hence
$$0\leq (\psi^*(\mathcal M)-E)^2=\mathcal M^2+E^2.$$
So $\mathcal M^2 \geq -E^2>0$.
On the other hand, the invertible sheaf $\mathcal M$ is also a quotient sheaf of $\tilde h^* \mathcal F_u$, which comes from a unitary representation,
hence it is semi-positive. This implies that the Kodaira dimension of $\mathcal M$ is $2$.
By \eqref{pi_*fenjie}, we get the following inclusion of sheaves,
\begin{equation*}
\wt L\otimes \mathcal M\subseteq \Omega_{\wt P}^1\left(\log (\Upsilon'+\wt R)\right).
\end{equation*}
As $2\wt L\equiv \wt R$ is effective,
the Kodaira dimension of $\wt L\otimes \mathcal M$ is also $2$, which is impossible by Bogomolov lemma (cf. \cite[Lemma\,7.5]{sakai80}).
Hence the image of $\varrho$ is an invertible subsheaf $\mathcal M$,
which is semi-positive since it is a quotient sheaf of a vector bundle coming from a unitary representation.

Finally, one can prove similarly that $\mathcal M\cdot D=0$ for any component $D\in\wt R$;
otherwise, $\wt L\otimes \mathcal M$ will be of Kodaira dimension two, which is impossible by Bogomolov lemma (cf. \cite[Lemma\,7.5]{sakai80}).
The proof is complete.
\end{proof}

Now we prove Theorem \ref{thmFtrivial'} and hence complete the proof of Theorem \ref{thmFtrivial}.
\begin{proof}[Proof of Theorem {\rm \ref{thmFtrivial'}}]
After a suitable base change, we may assume that
the fixed locus of the hyperelliptic involution $\sigma$ consists of $2g+2$ disjoint sections and possibly some isolated points;
actually, one achieves this by base change unbranched over $B\setminus\Delta$ as $f$ is semi-stable.
To prove our theorem, it suffices to show that $\mathcal F_u$ in \eqref{fujitadecomappendix} becomes trivial after a suitable base change.
We prove this by induction on $r=\rank \mathcal F_u$.

If $r=1$, then by \cite[\S\,4.2]{deligne71} or \cite[Theorem 3.4]{bar98},
we get that $\mathcal F_u$ is torsion in ${\rm Pic}^0(B)$.
So after a suitable finite \'etale base change, $\mathcal F_u$ will be trivial as required.

Now we assume that $r>1$.
If $\mathcal F_u=\mathcal F_u' \oplus \mathfrak L$ with $\rank \mathcal F_u'=r-1$ and $\rank\mathfrak L=1$,
then again by \cite[\S\,4.2]{deligne71} or \cite[Theorem 3.4]{bar98}, one gets that the pull-back of $\mathfrak L$ becomes trivial after a suitable finite \'etale base change;
hence our theorem follows from the induction.
Therefore, it is enough to prove that $\mathcal F_u$ contains a direct summand $\mathfrak L$ of rank one.

Let $s:\, B \to \wt P$ be a section of $\tilde h$ contained in $\wt R$, and $D=s(B)$, $j:\, D \hookrightarrow \wt P$ the inclusion.
By applying Lemma \ref{LR=0}, one obtains that
$\mathcal M\cdot D=0$ with $\mathcal M=\varrho\left(\tilde h^*\mathcal F_u\right)$, i.e., $\deg\mathcal O_D(\mathcal M)=0$.
Note that $s^*j^*\tilde h^*\mathcal F_u \cong \mathcal F_u$ and $s^*j^* \mathcal M = s^* \mathcal O_D(\mathcal M)$, where we view $s$ as a morphism from $B$ to $D$.
Hence we may view $\mathfrak L:=s^*\mathcal O_D(\mathcal M)$ as an invertible
sheaf on $B$, which is a quotient of $\mathcal F_u$ since $\mathcal M$ is a quotient of $\tilde h^*\mathcal F_u$.
Moreover, $\deg \mathfrak L = \deg \mathcal O_D(\mathcal M)=0$ since $s$ is an isomorphism.
As $\mathcal F_u$ comes from a unitary representation, $\mathcal F_u$ is poly-stable (cf. \cite{ns65}).
Thus $\mathcal F_u=\mathfrak L \oplus \mathcal F_u'$ contains a direct summand $\mathfrak L$ of rank one as required.
This completes the proof.
\end{proof}

\phantomsection
\addcontentsline{toc}{section}{Acknowledgements}
\noindent{\bf Acknowledgements.}
We would like to thank Hao Sun and Shengli Tan for
providing valuable discussions about Conjectures \ref{conjofxiao} and \ref{conjecturebs}.
We would also like to thank Xiaolei Liu, Jun Lu and Jinxing Xu for their interest.
Special thanks goes to Catanese for drawing our attention to \cite{cd14} and \cite{fujita78b}.
We are grateful to the anonymous referees.
Their valuable comments improved our proof and our writing significantly, which makes our paper more readable.

\enddocument